\documentclass[11pt]{amsart}

\usepackage[USenglish]{babel}

\usepackage{graphicx}
\usepackage{subcaption}

\usepackage{rank-2-roots}

\usepackage[pagebackref, colorlinks=true, linkcolor=black, citecolor=black, urlcolor=black]{hyperref}

\usepackage{xargs}                      
\usepackage[colorinlistoftodos,prependcaption 
]{todonotes}                                                                                                                 

\newcommandx{\todoN}[2][1=]{\todo[linecolor=red,backgroundcolor=white,bordercolor=red,#1]{#2}}			
\newcommandx{\todoA}[2][1=]{\todo[linecolor=green,backgroundcolor=white,bordercolor=green,#1]{#2}}			

\makeatletter
\def\namedlabel#1#2{\begingroup
    #2%
    \def\@currentlabel{#2}%
    \phantomsection\label{#1}\endgroup
}

\usepackage{comment}
\usepackage{calc}

\usepackage{amsmath, amsthm, amsfonts, amssymb}
\usepackage{enumerate}
\usepackage[shortlabels]{enumitem}
\usepackage{bbm}
\usepackage{mathrsfs}
\usepackage{amssymb}
\usepackage{mathtools}
\usepackage[all]{xy}
\usepackage{tikz}
\usepackage{tikz-cd}
\usetikzlibrary{matrix,arrows,decorations.pathmorphing}
\usepackage{xcolor}
\usepackage{bm}
	\usetikzlibrary{calc}
	\usetikzlibrary{trees}
	\tikzstyle{every picture}=[scale=.35,inner sep=0]
	\usepackage{color}

\newcommand{\rotsimeq}{\rotatebox[origin=c]{-90}{$\simeq$}}   

\usepackage{stmaryrd} 

\usepackage[normalem]{ulem} 

\newtheorem{thm}{Theorem}

\theoremstyle{definition}

\theoremstyle{theorem}
\newtheorem{theorem}{Theorem}[section]
\newtheorem{lemma}[theorem]{Lemma}
\newtheorem{proposition}[theorem]{Proposition}
\newtheorem{corollary}[theorem]{Corollary}

\theoremstyle{definition}
\newtheorem{definition}[theorem]{Definition}

\theoremstyle{remark}
\newtheorem{remark}[theorem]{Remark}

\numberwithin{equation}{section}

\makeatletter
\@namedef{subjclassname@2020}{\textup{MSC2020}}
\makeatother

\begin{document}

\title[Root lattices and invariant series for  plumbed $3$-manifolds]{Root lattices and invariant series\\ for  plumbed $3$-manifolds}

\author[A.H.~Moore]{Allison H.~Moore}
\address{Allison H.~Moore
\newline \indent Department of Mathematics \& Applied Mathematics
\newline \indent Virginia Commonwealth University, Richmond, VA 23284}
\email{moorea14@vcu.edu}

\author[N.~Tarasca]{Nicola Tarasca}
\address{Nicola Tarasca 
\newline \indent Department of Mathematics \& Applied Mathematics
\newline \indent Virginia Commonwealth University, Richmond, VA 23284}
\email{tarascan@vcu.edu}

\subjclass[2020]{
57K31 (primary), 
57K16, 
17B22 
(secondary)}
\keywords{Quantum invariants of 3-manifolds, plumbed $3$-manifolds, root systems, Kostant partition functions, $\mathrm{Spin}^c$-structures, false theta functions}

\begin{abstract}
We study formal series which are invariants of plumbed 3-manifolds twisted by root lattices. 
These series extend the BPS $q$-series $\widehat{Z}(q)$ recently defined in Gukov-Pei-Putrov-Vafa, Gukov-Manolescu, Park, and further refined in Ri. 
We show that the series $\widehat{Z}(q)$ is unique in an appropriate sense and decomposes as the average of related series which are themselves invariant under the five Neumann moves amongst plumbing trees. Explicit computations are presented in the case of Brieskorn spheres and a non-Seifert manifold.
\end{abstract}

\vspace*{-3.6pc}

\maketitle

\vspace{-2.pc}

\section*{Introduction}

An ongoing pursuit in quantum topology revolves around the categorification of Witten-Reshetikhin-Turaev (WRT) invariants for  links and 3-manifolds. Recent progress has been made through a physical definition of a new invariant series for  3-manifolds  
in Gukov-Pei-Putrov-Vafa \cite{gukov2020bps} and Gukov-Manolescu \cite{gukov2021two}. This series, usually denoted as $\widehat{Z}_a(q)$, requires the choice of
 a $\mathrm{Spin}^c$-structure $a$ on the 3-manifolds.
It is expected to arise as the Euler characteristic of a homology theory currently lacking a mathematical definition
 and is expected to converge to the WRT invariants in some appropriate limits \cite{gukov2017fivebranes, gukov2020bps}. 
A mathematical definition of $\widehat{Z}_a(q)$ is currently available only for negative-definite plumbed 3-manifolds \cite{gukov2021two}, and in this case the convergence to the WRT invariants has been recently proven in Murakami \cite{murakami2023proof}.
The series $\widehat{Z}_a(q)$ has been refined to include the datum of an arbitrary \textit{root lattice} by Park  \cite{park2020higher}, with the series from \cite{gukov2020bps, gukov2021two} coinciding with the $A_1$ case.

Here we show that $\widehat{Z}_a(q)$ is unique in an appropriate sense and decomposes as the average of related series which are themselves invariant.
For this, we start by constructing a general series starting from an initial input which includes the data used in \cite{park2020higher}. Namely we start from a reduced and refinable plumbing tree $\Gamma$ (reviewed in \S\S\ref{sec:redGamma} and \ref{sec:ref}), 
a {root lattice} $Q$, and a choice of a generalized $\mathrm{Spin}^c$-structure $a$ (which depends on the root lattice as  reviewed in \S\ref{sec:genspincstr}).
Moreover, we consider  ancillary inputs given by  a choice of a collection $\mathcal{P}$ of formal series and a set 
$S$ of assignments  of elements of the Weyl group to the vertex set of $\Gamma$. 
We denote our resulting series as $\mathsf{Y}\left(q\right)=\mathsf{Y}_{\mathcal{P},S,a}\left(q\right)$ (with $\Gamma$ and $Q$ omitted from the notation). In the statement that follows, $\Xi\subseteq W^{V(\Gamma)}$ is the set defined in \S\ref{sec:chamberass}.


\begin{thm}
\label{thm:finalthmintro}
For a reduced and refinable plumbing tree $\Gamma$, a root lattice $Q$, and a generalized $\mathrm{Spin}^c$-structure $a$, 
there exists a unique finite collection of formal series of type $\mathsf{Y}(q)$ that are invariant under the five Neumann moves amongst reduced plumbing trees. This collection consists of the series
\[
\mathsf{Y}_{S,a}\left(q\right) := \frac{1}{|S|}\sum_{\xi\in S} \mathsf{Y}_{\mathcal{K},\xi,a}\left(q\right) \qquad \mbox{for $S\subseteq \Xi$}
\]
with $\mathsf{Y}_{\Xi,a}\left(q\right) = \widehat{Z}_a(q)$.

If a series $\mathsf{Y}_{S,a}\left(q\right)$ for some $S\subseteq \Xi$ is also invariant under the action of the Weyl group, then $\mathsf{Y}_{S,a}\left(q\right) =  \widehat{Z}_a(q)$.
\end{thm}

The series $\mathsf{Y}_{\mathcal{P},S,a}\left(q\right)$ recovers the series $\widehat{Z}_a(q)$ when $\mathcal{P}$ is what we call the Kostant collection induced from the Kostant partition function of $Q$ as in \S\ref{sec:Kz} and $S=\Xi$.
We find the decomposition 
\[
\widehat{Z}_a(q) = \frac{1}{|\Xi|}\sum_{\xi\in \Xi} \mathsf{Y}_{\mathcal{K},\xi,a}\left(q\right)
\]
with each series $\mathsf{Y}_{\mathcal{K},\xi,a}\left(q\right)$ invariant under the five Neumann moves amongst reduced plumbing trees but not necessarily under the action of $W$.
For $Q=A_1$, the action of $W$ is the usual conjugation of $\mathrm{Spin}^c$-structures.

\smallskip

Our quest was initially inspired by Akhmechet-Johnson-Krushkal \cite{akhmechet2023lattice}, who show that for the  class of \textit{negative-definite plumbing trees} and $Q=A_1$, 
the series $\widehat{Z}_a(q)$ fits in an infinite family of invariant series. 
 As any two negative-definite plumbing tree presentations for a $3$-manifold are related via a series of two Neumann moves (\S\ref{sec:Neumann}), 
proving invariance is equivalent to checking \mbox{invariance} under the two Neumann moves. 
Moreover, the series $\widehat{Z}_a(q)$ is invariant with respect to conjugation of $\mathrm{Spin}^c$-structures. 
For an arbitrary series, the invariance under the two Neumann moves and under the conjugation of $\mathrm{Spin}^c$-structures imposes certain constraints on the series coefficients. 
It is shown in \cite{akhmechet2023lattice} that there are infinitely many solutions to such constraints, with the series $\widehat{Z}_a(q)$ giving only one such example.
Explicit computations in the case of  Brieskorn spheres are presented in \cite{liles2023infinite}.

In pursuit of constructing invariants for more general  3-manifolds, it is desirable to consider inputs from a larger class of plumbing graphs. When considering  more generally the case of arbitrary plumbing trees, any two plumbing tree presentations are related via a series of five Neumann moves, with three extra moves complementing the two moves from the negative-definite case. While the series $\widehat{Z}_a(q)$ (more precisely, a specific refinement presented for $Q=A_1$ in \cite{ri2023refined}) remains invariant (when it is well-defined, i.e., the plumbing tree is refinable as in \S\ref{sec:ref}) \cite{gukov2021two, park2020higher}, it is natural to ask 
whether $\widehat{Z}_a(q)$ is unique in this regard. In other words, we ask 
how many of the modifications of $\widehat{Z}_a(q)$ \`a la \cite{akhmechet2023lattice}
 satisfy the constraints given by the five Neumann moves and the action of the Weyl group on generalized $\mathrm{Spin}^c$-structures and thus remain invariant more generally for plumbing trees.

The answer to this question requires a study of the collections $\mathcal{P}$ used as input. 
Two key properties for such $\mathcal{P}$ emerge, referred to as admissibility and symmetry. 
We show that: 

\begin{thm}
\label{thm:uniquenessadm}
For a root lattice $Q$:
\begin{enumerate}[(i)]

\item There are only finitely many admissible collections, explicitly given in \eqref{eq:Padm}.

\item The Kostant collection is the unique admissible and symmetric collection, up to the action of the Weyl group $W$ \eqref{eq:WonP}.
\end{enumerate}
\end{thm}

Finally, we show that the invariance of the series $\mathsf{Y}(q)$ under the five Neumann moves is equivalent to $\mathcal{P}$ being admissible and $S\subseteq \Xi$. Additionally, the invariance of the series $\mathsf{Y}(q)$ also under the action of $W$ is equivalent to $\mathcal{P}$ being admissible and symmetric and $S= \Xi$ (see Theorem \ref{thm:qseriesinvariance}).
Thus we deduce Theorem \ref{thm:finalthmintro} from Theorem \ref{thm:uniquenessadm}.
%
%

Explicit computations in the case of Brieskorn spheres and a non-Seifert manifold are presented in \S\ref{sec:examples}.
In the former case, we show that  the series $\mathsf{Y}_{\mathcal{P},S,a}\left(q\right)$ for all $S\subseteq \Xi$ are identical and equal to $\widehat{Z}_a(q)$  for all root lattices, while in the latter case we show that the series $\widehat{Z}(q)$ for $Q=A_1$ decomposes as the average of two distinct invariant series of type $\mathsf{Y}\left(q\right)$.

We will present analogous results for plumbed knot complements and a gluing formula in the forthcoming \cite{MT2}. Additionally, we will show there how to remove the assumption that $\Gamma$ be refinable, at the expense of introducing the dependence of the series on a new variable.

\subsection*{Structure of the paper}
Root lattices, plumbings, Neumann moves, and $\mathrm{Spin}^c$-structures are reviewed in \S\ref{sec:background}. 
We define and study the admissible collections $\mathcal{P}$ in  \S\ref{sec:admser} 
and prove Theorem \ref{thm:uniquenessadm} there.
 The series $\mathsf{Y}\left(q\right)$ is defined in \S\ref{sec:qseries}.
For an admissible $\mathcal{P}$ and $S\subseteq \Xi$, the invariance of $\mathsf{Y}\left(q\right)$ is proved in \S\ref{sec:proofinv}.
The proof of Theorem \ref{thm:finalthmintro} is in \S\ref{sec:strong}.
Explicit computations are presented in \S\ref{sec:examples}.


\section{Notation and background}
\label{sec:background}

Here we review the necessary background on root lattices, plumbed $3$-manifolds and their homology, Neumann moves, reduced plumbing trees, and generalized  $\mathrm{Spin}^c$-structures. 

\subsection{Root lattices}
\label{sec:RootL}
We start by reviewing some basic  facts on root lattices that will be used throughout; we refer to \cite{MR1890629, MR0323842} for more details.

A \textit{root system} is a pair $(V, \Delta)$ where $V$ is a finite-dimensional Euclidean space over $\mathbb{R}$ with a positive-definite bilinear form $\langle , \rangle$, and 
$\Delta\subset V$ is a  finite subset of non-zero vectors, called \textit{roots}, such that:
\begin{enumerate}[(i)]
\item $\mathbb{R}\Delta=V$;
\item for $\alpha\in \Delta$, one has $n\alpha\in\Delta$ if and only if $n=\pm 1$;
\item $\Delta$ is closed under the reflections through the hyperplanes orthogonal to the roots; and 
\item for $\alpha,\beta\in \Delta$, one has $2\frac{\langle \alpha, \beta \rangle}{\langle \alpha, \alpha \rangle}\in \mathbb{Z}$.
\end{enumerate}

Let $Q$ be a \textit{root lattice}, that is, $Q=\mathbb{Z}\Delta$ for some root system $(V, \Delta)$. We will denote its rank as $r:=\mathrm{rank}(Q)$.
The corresponding \textit{weight lattice} $P$ is defined as
\[
P:= \left\{\lambda\in V \,\,\Bigg|\,\, 2\frac{\langle \lambda,\alpha\rangle}{\langle \alpha,\alpha\rangle}\in\mathbb{Z} \mbox{ for }\alpha\in \Delta \right\}.
\]

Select a set $\Delta^+\subset \Delta$ of \textit{positive roots}. 
This is the set of all roots  lying on a fixed side of a hyperplane in $V$ which does not contain any root.
The \textit{Weyl vector} $\rho\in P\cap \frac{1}{2}Q$ is  defined as the half-sum of the positive roots.

A root $\alpha\in \Delta^+$ is \textit{simple} if $\alpha$ cannot be written as the sum of two elements in~$\Delta^+$.
Simple roots form a basis for $V$. For  simple roots $\alpha_1, \dots, \alpha_r$, the \textit{fundamental weights} $\lambda_1,\dots, \lambda_r$ are elements of $P$ such that $2\frac{\langle \lambda_i,\alpha_j\rangle}{\langle \alpha_j,\alpha_j\rangle}=\delta_{i,j}$ for $i,j=1,\dots,r$. These also form a basis of $V$.

Let $W$ be the \textit{Weyl group} acting on $Q$. This is the group generated by reflections through the hyperplanes orthogonal to the roots.
For $w\in W$, the \textit{length} $\ell(w)$ of $w$ is defined as the minimum length of any expression of $w$ as product of such reflections.
This is also equal to the number of positive roots transformed by $w$ into negative roots.

Root lattices are classified by Dynkin diagrams. As an example, the root lattice  $Q=A_1$ is $\mathbb{Z}$ with bilinear form $\langle m,n\rangle = 2mn$ for $m,n\in\mathbb{Z}$. In this case, $\rho=\frac{1}{2}$ and $W=\mathbb{S}_2$ (the symmetric group on a set of size $2$). 

More generally, it will be convenient to have the following example in mind.
The root lattice  $Q=A_2$ is $\mathbb{Z}\alpha\oplus \mathbb{Z}\beta$ with $\langle \alpha,\alpha\rangle = \langle \beta,\beta\rangle=2$ and $\langle \alpha,\beta\rangle=-1$. In particular, the angle between $
\alpha$ and $\beta$ is $120^\circ$, and $Q$ is the vertex arrangement of the tiling of the Euclidean plane by equilateral triangles.
 In this case, 
$\Delta=\{\pm\alpha, \pm\beta,\pm(\alpha+\beta)\}$ and $W=\mathbb{S}_3$. 
For the set of positive roots $\Delta^+=\{\alpha, \beta,\alpha+\beta\}$, the Weyl vector is $\rho=\alpha+\beta$.

\subsection{Plumbings}
We will consider closed oriented $3$-manifolds that arise from the plumbing construction. Here we sketch the construction and set the notation; we refer to \cite{neumann1981calculus} and \cite[\S3.3]{MR4510934} for details. 

One starts from a plumbing graph $\Gamma$. This consists of a graph with some decorations: for each vertex, one has two integer numbers (called the \textit{Euler number} and the \textit{genus} of the vertex), and for each edge, one has a  sign. 
We assume throughout that $\Gamma$ is a \textit{tree} and that the genus of each vertex is \textit{zero}. Since $\Gamma$ has no cycles, one can assume that the sign on all edges is $+1$ \cite{neumann1981calculus} (but it will be beneficial to remember that edge signs can change; we will return to this in \S\S\ref{sec:Neumann}-\ref{sec:genspincstr}). Hence, for our plumbing trees we will only record the Euler number $m_v$ for each vertex $v$.

For a plumbing tree $\Gamma$, let $V(\Gamma)$ be its vertex set and $E(\Gamma)$  its edge set.
Choose an ordering of its vertices $v_1, \dots, v_s$, with $s=|V(\Gamma)|$, and let $m_1, \dots, m_s$ be the corresponding Euler numbers. An edge between vertices $v_i$ and $v_j$ will be denoted by $(i,j)\in E(\Gamma)$.

The \textit{framing matrix} $B$ determined by $\Gamma$ is the $s\times s$ symmetric matrix 
\[
B:= (B_{ij})_{i,j=1}^s \quad \mbox{ with }\quad
B_{ij}:=
\left\{
\begin{array}{ll}
m_i & \mbox{if $i=j$,}\\[5pt]
1 & \mbox{if $i\neq j$ and $(i, j)\in E(\Gamma)$,}\\[5pt]
0 & \mbox{otherwise.}
\end{array}
\right.
\]
(More generally, the entries $B_{ij}$ corresponding to the edges are defined to be equal to the edge signs.)
We denote by $\sigma = \sigma(B)$  the signature of $B$ and  $\pi= \pi(B)$  the number of its positive eigenvalues. One has $\sigma=2\pi-s$.

For the plumbing construction, one starts by assigning to each vertex $v$ of $\Gamma$ an oriented disk bundle over a real surface $E_v$ of genus equal to the genus decoration of $v$ (i.e., genus $0$ in our case), with the Euler number of the  bundle equal to  $m_v$.
Then one constructs a  $4$-manifold  $X=X(\Gamma)$ by gluing together such bundles according to the edge set $E(\Gamma)$.
Let ${M}={M}(\Gamma)$ be the boundary of $X$. This is a closed oriented $3$-manifold, called the \textit{plumbed $3$-manifold} constructed from $\Gamma$.
Alternatively, $M$ may be obtained by Dehn surgery on a framed link determined by $\Gamma$ consisting of unknots corresponding to the vertices of $\Gamma$, framings given by the corresponding Euler numbers, and with two unknots forming an Hopf link whenever the corresponding vertices in $\Gamma$ are joined by an edge.

\subsection{On the homology of the plumbed $3$-manifold}
\label{sec:homplumb}
The $4$-manifold $X$ has the same homotopy type of the space $E$ of the $s$ real surfaces $E_v$, i.e., $H_i\left( X; \mathbb{Z} \right)\cong H_i\left( E; \mathbb{Z} \right)$ for $i\geq 0$.
The homology of the $3$-manifold $M$ follows from Lefschetz duality, the Universal Coefficient Theorem, and the long exact sequence of the pair $(X,M)$.

Specifically, let 
\[
L:= H_2\left( X; \mathbb{Z} \right) \cong H_2\left( E; \mathbb{Z} \right)\cong \mathbb{Z}^s.
\]
 The last isomorphism is induced from the choice of an ordering of the real surfaces $E_v$ (or equivalently, the vertices of $\Gamma$). The natural  intersection pairing of $L\cong \mathbb{Z}^s$ is given by the framing matrix $B$.

By Lefschetz duality and the Universal Coefficient Theorem, the dual of the lattice $L$ is 
\[
L' = H^2\left( X; \mathbb{Z} \right)\cong H_2\left( X, {M}; \mathbb{Z} \right)\cong \mathbb{Z}^s.
\]
This is generated by the transversal disks $D_v$ to the surfaces $E_v$ at general points. Hence, the natural map $L\rightarrow L'$  in the bases $\{E_v\}_v$ and $\{D_v\}_v$ is given by the framing matrix $B$. 

In the following, we  assume that the pairing of $L$ is non-degenerate, i.e., $\det (B)\neq 0$. In this case, one has an inclusion of lattices $B\colon L\hookrightarrow L'$.
From the long exact sequence of the pair $(X,M)$, 
the boundary operator $L' \cong H_2\left( X, {M}; \mathbb{Z} \right) \rightarrow H_1\left( M; \mathbb{Z} \right)$ yields a short exact sequence
\[
L'/BL \hookrightarrow H_1\left( M; \mathbb{Z} \right) \twoheadrightarrow H_1\left( X; \mathbb{Z} \right).
\]
From the assumption that $\Gamma$ is a tree and that all surfaces $E_v$ have genus $0$, we deduce the vanishing $H_1\left( X; \mathbb{Z} \right) \cong H_1\left( E; \mathbb{Z} \right)\cong 0$. Hence, one has
\[
H_1\left( M; \mathbb{Z} \right) \cong L'/BL \cong \mathbb{Z}^s/B\mathbb{Z}^s.
\]
In particular, our assumptions imply that $M$ is a rational homology sphere, i.e., $H_1\left( M; \mathbb{Q} \right)=0$.

When $\det (B)\neq 0$, the framing matrix $B$ is invertible over $\mathbb{Q}$, and
the induced bilinear pairing on $L'$ is 
\[
\langle, \rangle \colon L' \times L' \rightarrow \mathbb{Q}, \qquad (v, w) \mapsto v^t B^{-1}w.
\]
This pairing is induced from $B$ since for $x,y\in L$, one has $\langle Bx, By \rangle = x^t B y$, thus recovering the pairing of $x$ and $y$ in $L$.

For a root lattice $Q$  and a lattice $L'$ as above, the induced bilinear pairing on $L'\otimes_{\mathbb{Z}} Q$ is defined by
\[
\langle, \rangle \colon L'\otimes_{\mathbb{Z}} Q \times L'\otimes_{\mathbb{Z}} Q \rightarrow \mathbb{Q}, \qquad (v\otimes \alpha, w\otimes \beta) \mapsto \langle v, w \rangle\, \langle \alpha, \beta \rangle.
\]
Here the pairing $ \langle v, w \rangle$ is in $L'$ and the pairing $\langle \alpha, \beta \rangle$ is in $Q$.
This extends by linearity as follows. For $a,b\in L'\otimes_{\mathbb{Z}} Q\cong \mathbb{Z}^s\otimes_{\mathbb{Z}} Q$, write $a=(a_1,\dots, a_s)$ and $b=(b_1,\dots, b_s)$, with $a_i,b_i\in Q$, for $i=1,\dots,s$.
Then the pairing is
\begin{equation}
\label{eq:pairingL'Q}
\langle a, b \rangle = \sum_{i=1}^s \sum_{j=1}^s \left( B^{-1}\right)_{ij} \langle a_i, b_j\rangle.
\end{equation}

\subsection{Refinable plumbings}
\label{sec:ref}
A plumbing tree $\Gamma$  is \textit{negative definite} if the framing matrix $B$ is negative definite.

A plumbing tree $\Gamma$ is \textit{weakly negative definite} if the framing matrix $B$ is invertible over $\mathbb{Q}$ and 
$B^{-1}$ is negative definite
 on the subspace of $\mathbb{Z}^s$ spanned by the vertices of $\Gamma$ of degree at least $3$. If $\Gamma$ has no vertex of degree at least $3$, then one simply requires that $B$ be invertible over $\mathbb{Q}$.

A plumbed $3$-manifold $M$ is \textit{negative definite} (respectively, \textit{weakly negative definite}) if $M$ may be constructed from some negative-definite (resp., weakly negative-definite) plumbing tree, up to an orientation preserving homeomorphism. 

One defines a plumbing tree or a plumbed $3$-manifold to be \textit{weakly positive definite} similarly.

We say that a plumbing tree $\Gamma$ is \textit{refinable} if up to Neumann moves, $\Gamma$ is either weakly negative definite or weakly positive definite. 

All weakly negative-definite plumbed $3$-manifolds are in fact negative definite, that is,  every weakly negative-definite plumbing tree can be transformed into a negative definite one by a sequence of Neumann moves \cite[Thm A.1]{harichurn2025delta}. 
In particular, all such $3$-manifolds arise as links of isolated singularities on complex normal surfaces, up to an orientation preserving homeomorphism.

\begin{figure}[t]
\hspace{-0.8cm}
\begin{subfigure}[t]{0.26\textwidth}
\[
\begin{array}{c}
\begin{tikzpicture}[baseline={([yshift=0ex]current bounding box.center)}]
      \path(0,0) ellipse (2 and 2);
      \tikzstyle{level 1}=[counterclockwise from=150,level distance=15mm,sibling angle=30]
      \node [draw,circle, fill, inner sep=1.5, label={[label distance=10]90:$m_1+\epsilon1$}] (A1) at (180:4) {}
            child {node [label=: {}]{}}
	    child {node [label=: {}]{}}
    	    child {node [label=: {}]{}};
      \tikzstyle{level 1}=[counterclockwise from=30,level distance=15mm,sibling angle=-60]
      \node [draw,circle,fill, inner sep=1.5, label={[label distance=10]90:$m_2+\epsilon1$}] (A2) at (0:4) {}
	   child {node [label=: {}]{}}
	   child {node [label=: {}]{}};
      \tikzstyle{level 1}=[counterclockwise from=150,level distance=15mm,sibling angle=30]
      \node [draw,circle, fill, inner sep=1.5, label={[label distance=10]90:$\epsilon1$}] (A0) at (0:0) {};
      \path (A1) edge []  node[midway, label={[label distance=10]-90:}]{} (A0);
       \path (A2) edge []  node[midway, label={[label distance=10]-90:}]{} (A0);
    \end{tikzpicture}
\\
\rotsimeq
    \\
\begin{tikzpicture}[baseline={([yshift=0ex]current bounding box.center)}]
      \path(0,0) ellipse (2 and 2.5);
      \tikzstyle{level 1}=[counterclockwise from=150,level distance=15mm,sibling angle=30]
      \node [draw,circle, fill, inner sep=1.5, label={[label distance=8]90:$m_1$}] (A1) at (180:2) {}
            child {node [label=: {}]{}}
	    child {node [label=: {}]{}}
    	    child {node [label=: {}]{}};
      \tikzstyle{level 1}=[counterclockwise from=30,level distance=15mm,sibling angle=-60]
      \node [draw,circle,fill, inner sep=1.5, label={[label distance=8]90:$m_2$}] (A2) at (0:2) {}
	   child {node [label=: {}]{}}
	   child {node [label=: {}]{}};
      \path (A1) edge []  node[midway, label={[label distance=10]-90:}]{} (A2);
    \end{tikzpicture}
\end{array}
\]
\caption*{(A$\epsilon$)}
\end{subfigure}
\qquad\quad
\begin{subfigure}[t]{0.26\textwidth}
\[
\begin{array}{c}
\begin{tikzpicture}[baseline={([yshift=0ex]current bounding box.center)}]
      \path(0,0) ellipse (2 and 2);
      \tikzstyle{level 1}=[counterclockwise from=150,level distance=15mm,sibling angle=30]
      \node [draw,circle, fill, inner sep=1.5, label={[label distance=10]90:$m_1+\epsilon1$}] (A1) at (180:4) {}
            child {node [label=: {}]{}}
	    child {node [label=: {}]{}}
    	    child {node [label=: {}]{}};
      \tikzstyle{level 1}=[counterclockwise from=150,level distance=15mm,sibling angle=30]
      \node [draw,circle, fill, inner sep=1.5, label={[label distance=10]90:$\epsilon1$}] (A0) at (0:0) {};
      \path (A1) edge []  node[midway, label={[label distance=10]-90:}]{} (A0);
    \end{tikzpicture}
\\
\rotsimeq
    \\
\begin{tikzpicture}[baseline={([yshift=0ex]current bounding box.center)}]
      \path(0,0) ellipse (2 and 2.5);
      \tikzstyle{level 1}=[counterclockwise from=150,level distance=15mm,sibling angle=30]
      \node [draw,circle, fill, inner sep=1.5, label={[label distance=8]90:$m_1$}] (A1) at (0:0) {}
            child {node [label=: {}]{}}
	    child {node [label=: {}]{}}
    	    child {node [label=: {}]{}};
    \end{tikzpicture}
\end{array}
\]
\caption*{(B$\epsilon$)}
\end{subfigure}
\begin{subfigure}[t]{0.26\textwidth}
\[
\begin{array}{c}
\begin{tikzpicture}[baseline={([yshift=0ex]current bounding box.center)}]
      \path(0,0) ellipse (2 and 2);
      \tikzstyle{level 1}=[counterclockwise from=150,level distance=15mm,sibling angle=30]
      \node [draw,circle, fill, inner sep=1.5, label={[label distance=10]90:$m_1$}] (A1) at (180:4) {}
            child {node [label=: {}]{}}
	    child {node [label=: {}]{}}
    	    child {node [label=: {}]{}};
      \tikzstyle{level 1}=[counterclockwise from=30,level distance=15mm,sibling angle=-60]
      \node [draw,circle,fill, inner sep=1.5, label={[label distance=10]90:$m_2$}] (A2) at (0:4) {}
	   child {node [label={[label distance=5]0:}]{}}
	   child {node [label={[label distance=5]0:}]{}};
      \tikzstyle{level 1}=[counterclockwise from=150,level distance=15mm,sibling angle=30]
      \node [draw,circle, fill, inner sep=1.5, label={[label distance=10]90:$0$}] (A0) at (0:0) {};
      \path (A1) edge []  node[midway, label={[label distance=10]-90:}]{} (A0);
       \path (A2) edge []  node[midway, label={[label distance=10]-90:}]{} (A0);
    \end{tikzpicture}
\\
\rotsimeq
    \\
\begin{tikzpicture}[baseline={([yshift=0ex]current bounding box.center)}]
      \path(0,0) ellipse (2 and 2.5);
      \tikzstyle{level 1}=[counterclockwise from=150,level distance=15mm,sibling angle=30]
      \node [draw,circle, fill, inner sep=1.5, label={[label distance=10]90:$m_1+m_2$}] (A1) at (0:0) {}
            child {node [label=: {}]{}}
	    child {node [label=: {}]{}}
    	    child {node [label=: {}]{}};
      \tikzstyle{level 1}=[counterclockwise from=30,level distance=15mm,sibling angle=-60]
      \node [draw,circle,fill, inner sep=1.5, label={[label distance=8]90:}] (A2) at (0:0) {}
	   child {node [label={[label distance=10]0:}]{}}
	   child {node [label={[label distance=10]0:}]{}};
    \end{tikzpicture}
\end{array}
\]
\caption*{(C)}
\end{subfigure}

\caption{The five Neumann moves on plumbing trees. 
Here, \mbox{$\epsilon \in \{+, -\}$.}
}
\label{fig:Neumann}
\end{figure}

\subsection{Neumann moves}
\label{sec:Neumann}
Neumann showed that if two plumbing graphs represent the same $3$-manifold up to orientation-preserving diffeomorphism, then they are related by a finite sequence of combinatorial moves \cite[Thm~3.2]{neumann1981calculus}. The only such moves between two plumbing \textit{trees} are the five moves given in Figure \ref{fig:Neumann} and their inverses.

%

Note that the moves (A$-$) and (B$-$) from Figure \ref{fig:Neumann} preserve the negative-definite property of the plumbing trees.
However, the other three Neumann moves from Figure \ref{fig:Neumann}  do not necessarily preserve the weakly negative-definite property of the plumbing trees, see \cite[Ex.~4.2]{ri2023refined}. In particular, a plumbing tree for a weakly negative-definite plumbed $3$-manifold may not necessarily be weakly negative definite, but it may become so after a sequence of the Neumann moves  from Figure \ref{fig:Neumann}.

We emphasize that orientation preserving homeomorphisms between pairs consisting of a plumbed $3$-manifold and a (generalized) $\mathrm{Spin}^c$-structure have not been classified yet, see \cite[Appendix A]{akhmechet2024knot}. 
Also, while the existence of an orientation preserving homeomorphism between $3$-manifolds constructed from plumbing \textit{forests} is established in \cite{neumann2006invariant}, no such result is known when one reduces to plumbing trees alone. We will content ourselves with verifying invariance under the Neumann moves  from Figure \ref{fig:Neumann} and the action of the Weyl group.

\begin{remark}
\label{rmk:columnspace}
For each move, we will use the following observation about the framing matrices. 
Let $\Gamma$ and $\Gamma_\circ$ be the bottom and top plumbing trees, respectively, and let $B$ and $B_\circ$ be the corresponding framing matrices. 
A direct computation shows that the column space of $B$ is isomorphic to a subspace of the column space of $B_\circ$.  
We show this in the case of move (A$+$), with the case of the other moves being similar. 
Assume first that $\Gamma$ has only two vertices. Then $B$ and $B_\circ$ are
\[
B= \left( 
\begin{array}{cc}
m_1 & 1\\
1 & m_2
\end{array}
\right)
\qquad \mbox{and} \qquad
B_\circ= \left( 
\begin{array}{ccc}
m_1+1 & 1 & 0\\
1 & 1 & -1\\
0 & -1 & m_2+1
\end{array}
\right).
\]
After a column operation, $B_\circ$ becomes
\[
\left( 
\begin{array}{ccc}
m_1 & 1 & 1\\
0 & 1 & 0\\
1 & -1 & m_2
\end{array}
\right)
\]
from which it is clear that the column space of $B$ is isomorphic to a subspace of the column space of $B_\circ$.
Note that the negative edge signs appearing as the two coefficients $-1$ of $B_\circ$ can be undone after changing the orientation of the subspace spanned by the vertex labelled by $m_2$. Indeed, this entails multiplying by $-1$ the last row and last column of $B_\circ$.
For the case when $\Gamma$ has more vertices, a similar block form argument applies. 
\end{remark}

\begin{remark}
A composition of Neumann moves does not necessarily preserve the orientation of the $3$-manifold. For instance, consider the plumbing tree for the lens space $L(-2,1)$ consisting of a single vertex of weight $-2$. After performing the Neumann move (B$+$) twice, apply move (C) to obtain the plumbing tree for $L(2,1)$ with a single vertex of weight $2$.
This is an instance of orientation-reversing self-homeomorphisms of lens spaces, which are classified in \cite{reidemeister1935homotopieringe}, see also \cite[Thm A.4.24]{przytycki2023lectures}.
\end{remark}

\subsection{Generalized $\mathrm{Spin}^c$-structures}
\label{sec:genspincstr}

Here we review the space of generalized $\mathrm{Spin}^c$-structures for a plumbed $3$-manifold $M$ and a root lattice $Q$. This space appeared in \cite{park2020higher}.
A generalized $\mathrm{Spin}^c$-structure will be an input of the $q$-series defined in~\S\ref{sec:qseries}.

Recall from \S\ref{sec:homplumb} that the choice of an ordering of the vertices of $\Gamma$ induces an isomorphism $L'\cong \mathbb{Z}^s$.
Let
\[
T_E := \left(2-\deg(v_1), \dots, 2-\deg(v_s) \right) \in \mathbb{Z}^{s} \cong L' 
\]
and
\begin{equation}
\label{eq:delta}
\delta:= T_E \otimes 2\rho \in L' \otimes_{\mathbb{Z}} Q
\end{equation}
where $\rho$ is the Weyl vector  as in \S\ref{sec:RootL}.

The \textit{space of generalized $\mathrm{Spin}^c$-structures} on $M$ for the root lattice $Q$ is
\begin{equation}
\label{eq:BQM}
\mathsf{B}_Q(M) := \frac{\delta+2L' \otimes_{\mathbb{Z}} Q}{2BL\otimes_{\mathbb{Z}} Q}.
\end{equation}
For $Q=A_1$, this is simply 
\[
\mathrm{Spin}^c(M) \cong \frac{\delta+2L' }{2BL},
\]
the space of $\mathrm{Spin}^c$-structures on~$M$ \cite[\S 6.10]{MR4510934}.
Thus $\mathsf{B}_Q(M)$ generalizes the space of $\mathrm{Spin}^c$-structures for an arbitrary root lattice $Q$. As for the case $Q=A_1$, the space $\mathsf{B}_Q(M)$ is affinely isomorphic to
\begin{equation}
\label{eq:H1MQ}
H_1(M; Q)\cong \frac{L' \otimes_{\mathbb{Z}} Q}{BL\otimes_{\mathbb{Z}} Q}.
\end{equation}
The Weyl group $W$ naturally acts component-wise on $L' \otimes_{\mathbb{Z}} Q$, and this induces an action of $W$ on $\mathsf{B}_Q(M)$:
\[
w\colon \mathsf{B}_Q(M) \rightarrow \mathsf{B}_Q(M), \qquad [a]\mapsto [w(a)], \qquad \mbox{for $w\in W$}.
\]

\begin{proposition}
\label{prop:BQMinvariance}
For a plumbed $3$-manifold $M$ and a root lattice $Q$,
the set $\mathsf{B}_Q(M)$  and the Weyl group action on it are invariant under the Neumann moves in Figure \ref{fig:Neumann}.
\end{proposition}

Proposition \ref{prop:BQMinvariance} follows from the next Proposition \ref{prop:BQMinvariance2}. This will be used in the proof of Theorem~\ref{thm:seriesinvariancef}.

For each move, we use the notation $B\colon L\hookrightarrow L'$ and $\delta$ defined as above for the terms related to the bottom plumbing tree $\Gamma$, and the notation
\mbox{$B_\circ \colon L_\circ \hookrightarrow L_\circ'$} and $\delta_\circ$ for the corresponding terms related to the top plumbing  tree $\Gamma_\circ$.

For each move,  we define a function 
\begin{equation*}
R\colon L'\otimes_{\mathbb{Z}} Q \rightarrow L'_\circ \otimes_{\mathbb{Z}} Q, \qquad a\mapsto R(a)
\end{equation*}
such that the induced map
\begin{equation}
\label{eq:isoBQM}
\frac{\delta+2L' \otimes_{\mathbb{Z}} Q}{2BL\otimes_{\mathbb{Z}} Q} \longrightarrow 
\frac{\delta_\circ+2L_\circ' \otimes_{\mathbb{Z}} Q}{2B_\circ L_\circ\otimes_{\mathbb{Z}} Q}, \qquad
[a] \mapsto [R(a)]
\end{equation}
is a bijection of sets and is equivariant with respect to the action of the Weyl group $W$.
Note that for each move, the column space of $B$ is isomorphic to a subspace of the column space of $B_\circ$, see Remark \ref{rmk:columnspace}. It follows that for each representative $a$ of a generalized $\mathrm{Spin}^c$-structure for $\Gamma$, there is a corresponding affine space of generalized $\mathrm{Spin}^c$-structures for $\Gamma_\circ$, and $R(a)$ is required to be in such a~space.

We proceed by defining the function $R$ for each move. For this, we first choose an order of the vertices of $\Gamma$ and a compatible order of the vertices of $\Gamma_\circ$. This induces isomorphisms $L'\cong \mathbb{Z}^s$ and $L'_\circ\cong \mathbb{Z}^{s_\circ}$, where $s$ and $s_\circ$ are the ranks of $L'$ and $L'_\circ$, respectively. 

\begin{remark}
\label{rmk:leftright}
The choice of an order of the vertices of $\Gamma$ and a compatible order of the vertices of $\Gamma_\circ$ allows one to distinguish for each Neumann move the parts of the tree that are on the left and on the right of each vertex.
\end{remark}

\smallskip

Consider the Neumann move (A$\epsilon$) from Figure \ref{fig:Neumann} with $\epsilon\in\{+,-\}$. 
For $a\in L' \otimes_{\mathbb{Z}} Q$, write $a=(a_1,a_2)$ with subtuple $a_1$ corresponding to the vertices of $\Gamma$ consisting of the vertex labeled by $m_1$ and all vertices on its left, and subtuple $a_2$ corresponding to the vertices of $\Gamma$ consisting of the vertex labeled by $m_2$ and all vertices on its right.
Define 
\begin{equation}
\label{eq:Ra}
R\colon L'\otimes_{\mathbb{Z}} Q \rightarrow L'_\circ \otimes_{\mathbb{Z}} Q, \qquad
(a_1, a_2)\mapsto (a_1, 0, -\epsilon a_2)
\end{equation}
 with the $0$ entry corresponding to the added vertex in $\Gamma_\circ$.
 Recall from \S\ref{sec:homplumb}  that all edges of plumbing graphs have a sign, which determines the corresponding gluing, and in the case of plumbing trees one can assume that all edge signs are equal \cite{neumann1981calculus}. 
When $\epsilon =+$, the Neumann move (A$+$) involves the change of an edge sign of  $\Gamma_\circ$. This change of the edge sign can be undone after changing the orientation of the subspace of $L'\otimes_{\mathbb{Z}} Q$ corresponding to one side of the added vertex in $\Gamma_\circ$. Thus the minus sign multiplying $a_2$ in the formula for $R$.

\smallskip

Next, consider the Neumann move from Figure \ref{fig:Neumann}(B$\epsilon$) with $\epsilon\in\{+,-\}$. 
For $a\in L' \otimes_{\mathbb{Z}} Q$, write $a=(a_\sharp,a_1)$ with entry $a_1$ corresponding to the vertex of $\Gamma$ labeled by $m_1$, and subtuple $a_\sharp$ corresponding to all other vertices of~$\Gamma$.
Define 
\begin{equation}
\label{eq:RaB}
R \colon L'\otimes_{\mathbb{Z}} Q \rightarrow L'_\circ \otimes_{\mathbb{Z}} Q, \qquad
(a_\sharp, a_1)\mapsto (a_\sharp, a_1+2\rho,  \epsilon 2 \rho)
\end{equation}
 where the entry $\epsilon 2\rho$ corresponds to the added vertex in $\Gamma_\circ$.

\smallskip

Finally, consider the Neumann move (C) from Figure \ref{fig:Neumann}. 
Let $v_0$  be the vertex in $\Gamma$ labeled by $m_1+m_2$, and let $v_1$, $v'_0$, and $v_2$ be the vertices in $\Gamma_\circ$ labeled by $m_1$, $0$, and $m_2$, respectively.
For $a\in L' \otimes_{\mathbb{Z}} Q$, write $a=(a_\sharp, a_0, a_\flat)$ with entry $a_0$ corresponding to the vertex $v_0$ in $\Gamma$,
subtuple $a_\sharp$ corresponding to all vertices in $\Gamma$ equivalent to the vertices of $\Gamma_\circ$ on the left of $v_1$, and subtuple $a_\flat$ corresponding to all vertices in $\Gamma$ equivalent to the vertices of $\Gamma_\circ$ on the right of $v_2$.
For $\beta\in Q$,  define 
\begin{equation}
\label{eq:RaC}
R=R_\beta\colon L'\otimes_{\mathbb{Z}} Q \rightarrow L'_\circ \otimes_{\mathbb{Z}} Q, \qquad
(a_\sharp, a_0, a_\flat)\mapsto (a_\sharp,  a_0+\beta, 0, \beta, -a_\flat)
\end{equation}
 where the entries $a_0+\beta$, $0$, and $\beta$ correspond to the vertices $v_1$, $v'_0$, and $v_2$ in $\Gamma_\circ$, respectively. Assume that $\beta$ is chosen as follows
 \begin{equation}
\label{eq:gammadef}
 \beta = \left\{
 \begin{array}{ll}
2\rho & \mbox{if $\deg(v_1)\equiv \deg(v_2)$ mod $2$,}\\[5pt]
0 & \mbox{otherwise.}
 \end{array}
 \right.
 \end{equation}
 This choice of $\beta$ will allow one to verify \eqref{eq:Rdeltaintodelta}.

\begin{proposition}
\label{prop:BQMinvariance2}
For each Neumann move  in Figure \ref{fig:Neumann}, the map $R$ as defined in \eqref{eq:Ra}--\eqref{eq:gammadef} induces a bijection of the sets $\mathsf{B}_Q(M)$ as in \eqref{eq:isoBQM}.

The induced bijection is equivariant with respect to changes of the order of the vertices of the plumbing trees and the action of the Weyl group. 
\end{proposition}

\begin{proof}
For each Neumann move  in Figure \ref{fig:Neumann}, it is immediate to verify that $R$ is injective. Moreover, one has
\begin{equation}
\label{eq:Rdeltaintodelta}
R\left(\delta+2L' \otimes_{\mathbb{Z}} Q\right) \subseteq \delta_\circ+2L_\circ' \otimes_{\mathbb{Z}} Q.
\end{equation}
For instance, let us verify this for the Neumann move (C). Select 
\[
\ell\in \delta+2L' \otimes_{\mathbb{Z}} Q, \quad\mbox{ and write } \quad\ell=(\ell_\sharp, \ell_0, \ell_\flat).
\]
This implies $\ell_0\in (2-\deg v_0)2\rho+2Q$. Then in order to have 
\[
R_\beta(\ell)= (\ell_\sharp,  \ell_0+\beta, 0, \beta, -\ell_\flat)\in \delta_\circ+2L_\circ' \otimes_{\mathbb{Z}} Q, 
\]
one needs
\[
\ell_0+\beta \in (2-\deg v_1)2\rho+2Q \quad \mbox{and}\quad \beta \in (2-\deg v_2)2\rho+2Q.
\]
One has $\deg v_0\equiv \deg v_1$ mod $2$ if and only if $\deg v_1 \not\equiv \deg v_2$ mod $2$.
Hence both of these conditions are implied by the choice of $
\beta$ in \eqref{eq:gammadef}.

The induced map \eqref{eq:isoBQM} is thus well-defined and injective.
Recall that  the column space of $B$ is isomorphic to a subspace of the column space of $B_\circ$, see Remark \ref{rmk:columnspace}.
Hence the surjectivity of the induced map \eqref{eq:isoBQM} follows by a direct analysis of the extra column space of $B_\circ$.
Finally, the equivariance with respect to changes of the order of the vertices of the plumbing trees and the action of the Weyl group  follows immediately.
\end{proof}

\subsection{Reduced plumbing trees}
\label{sec:redGamma}
We will use reduced plumbing trees as in \cite{ri2023refined}. These are defined as follows.
Let $\Gamma$ be a plumbing tree. 
First, define a \textit{branch} of $\Gamma$ to be a path in $\Gamma$ 
connecting a vertex of degree at least three to a vertex of degree one through a sequence of degree-$2$ vertices. 
Define a branch to be \textit{contractible} if the branch can be contracted down to a single vertex by a sequence of the Neumann moves from Figure \ref{fig:Neumann}.

Contractible branches can be characterized using continued fractions. 
For integers $a_0, \dots, a_n$ with $a_n\neq 0$, consider the \textit{negative continued fraction}
\[
\llbracket a_0, \dots, a_n \rrbracket := a_0 -\dfrac{1}{a_1 - \dfrac{1}{\dots -\dfrac{1}{a_n}}} \in \mathbb{Q}.
\] 
When $a_n=0$, define 
\begin{equation}
\label{eq:contfractend0}
\llbracket a_0, \dots, a_{n-2}, a_{n-1}, 0 \rrbracket:= \llbracket a_0, \dots, a_{n-2} \rrbracket.
\end{equation}
For a path $\Gamma_{wv}$ from a vertex $w$ to a vertex $v$ through a sequence of degree-$2$ vertices, let $m_w, u_1, \dots, u_n, m_v$ be the weights of the vertices from $w$ to $v$ taken in this order. 

\begin{lemma}
Changing $\Gamma_{wv}$ through a sequence of paths via the Neumann moves from Figure \ref{fig:Neumann}
preserves $\llbracket m_w, u_1, \dots, u_n, m_v \rrbracket$. 
\end{lemma}

\begin{proof}
This follows from the three identities
\begin{align*}
\llbracket \dots, a \pm 1, \pm1, b\pm 1, \dots \rrbracket &= \llbracket \dots, a , b, \dots \rrbracket,\\
\llbracket \dots, a \pm1, \pm 1 \rrbracket &=  \llbracket \dots, a \rrbracket,\\
\llbracket \dots, a, 0, b,\dots \rrbracket &= \llbracket \dots, a+b,\dots \rrbracket
\end{align*}
corresponding to the Neumann moves from Figure \ref{fig:Neumann}.
For these identities, see e.g., \cite[pg.~57]{hirasawa2006genera}.
Note that \eqref{eq:contfractend0} is compatible with Neumann's splitting move  \cite[Prop.~2.2.(3)]{neumann1981calculus}.
\end{proof}

Consequently, one has:

\begin{lemma}
\label{lemma:contractiblepath}
The path $\Gamma_{wv}$ is contractible if and only if $\llbracket m_w, u_1, \dots, u_n, m_v \rrbracket$ is an integer. In this case, 
$\llbracket m_w, u_1, \dots, u_n, m_v \rrbracket$ equals the weight of the vertex resulting from the contraction.
\end{lemma}

In particular, this statement applies when the path $\Gamma_{wv}$ is a branch as above.

A vertex $v$ of $\Gamma$ is defined to be \textit{reducible} if $v$ has degree at least $3$ but, after contracting all contractible branches incident to $v$, the degree of $v$ drops down to $1$ or $2$.

Finally, define  $\Gamma$ to be \textit{reduced} if $\Gamma$ has no reducible vertices. 
Any plumbing tree can be reduced  via a sequence of the Neumann moves from Figure \ref{fig:Neumann}.
Note that reducing a reducible vertex to a vertex of degree $1$ or $2$ via a sequence of Neumann moves may yield a new reducible vertex. For this, the tree $\Gamma$ becomes reduced after \textit{repeatedly} reducing all reducible vertices via a sequence of Neumann moves.
Moreover, one has:

\begin{lemma} 
\label{lemma:weakcontract}
By removing contractible branches, a weakly negative-definite plumbing tree becomes reduced while remaining weakly negative definite.
\end{lemma}

\begin{proof}
We  argue that contracting a branch preserves the property of being weakly negative definite. A branch may be contracted by a sequence of the Neumann moves from Figure \ref{fig:Neumann} of type (A$\pm$), (B$\pm$), and those moves of type (C) where at most one of the two vertices labelled by $m_1$ and $m_2$ has degree at least $3$. 
(Moves of type (C) where the two vertices labelled by $m_1$ and $m_2$ have both degree at least $3$ are not necessary to contract branches. 
In fact, these moves do not necessarily preserve the weakly negative-definite property of the plumbing trees, see \cite[Ex.~4.2]{ri2023refined}. Hence we avoid using them in this argument.)

For each one of these Neumann moves, let $\Gamma$ and $\Gamma_\circ$ be the bottom and top plumbing graphs, respectively. We argue that if one of them is weakly negative definite, so is the other one. As in \S\ref{sec:genspincstr}, let $B$ and $B_\circ$ be the framing matrices of $\Gamma$ and $\Gamma_\circ$, respectively, and let $s$ and $s_\circ$ be their ranks, respectively. 
Let $H\subset Q^s$ and $H_\circ\subset Q^{s_\circ}$ be the subspaces spanned by the vertices of degree~\mbox{$\geq 3$} in $\Gamma$ and $\Gamma_\circ$, respectively. 
If the Neumann move under consideration reduces a reducible vertex to a vertex of degree $1$ or $2$ in $\Gamma$, then quotient $H_\circ$ by the linear subspace corresponding to that reducible vertex of degree $3$ in $\Gamma_\circ$, and denote this quotient still by $H_\circ$. Thus we may assume that $H$ and $H_\circ$ have the same rank.

We proceed to construct a map $R\colon Q^s\rightarrow Q^{s_\circ}$ which induces a linear isomorphism $H\cong H_\circ$.
For the moves (A$\pm$), consider the map $R$ from \eqref{eq:Ra}; for the moves (B$\pm$), consider the map
\[
R \colon Q^s \rightarrow Q^{s+1}, \qquad
(a_\sharp, a_1)\mapsto (a_\sharp, a_1, 0)
\]
with notation as in \eqref{eq:RaB}; and for the move (C), consider the map $R=R_\beta$ from \eqref{eq:RaC} with $\beta=0$. It is immediate to see that for each move, the map $R$ so defined induces an isomorphism $H\cong H_\circ$. 

Moreover, for each $\ell\in H$, a direct computation shows that $\langle \ell, \ell\rangle= \langle R(\ell), R(\ell)\rangle$, with the pairings defined by the matrices $B^{-1}$ and $B_\circ^{-1}$ as in \eqref{eq:pairingL'Q}, respectively. For moves (A$\pm$) and (C), this is a special case of a more general computation later done in \eqref{eq:RRllA-}, \eqref{eq:RRllA+}, \eqref{eq:RRllC}; the case of moves (B$\pm$) follows similarly.
Hence the statement.
\end{proof}

We will use the following result from \cite{ri2023refined} 

\begin{proposition}[{\cite[Prop.~3.4]{ri2023refined}}]
Any two reduced plumbing trees are related by a sequence of the Neumann moves from Figure \ref{fig:Neumann}
if and only if they are related by a sequence of those Neumann moves from Figure \ref{fig:Neumann} which do not create any reducible vertices.
\end{proposition}

We include a proof for the benefit of the reader.

\begin{proof}
Let $\Gamma$ and $\Gamma'$ be two reduced plumbing trees and assume that there exists a sequence 
\[
\mathcal{S}:=(\Gamma=\Gamma_1 \rightarrow \Gamma_2\rightarrow \dots \rightarrow \Gamma_n=\Gamma')
\]
of plumbing trees related by Neumann moves. Assume that one of these plumbing trees is not reduced, i.e., it contains at least one reducible vertex.
Select a $\Gamma_i$ in $\mathcal{S}$ with the highest number of reducible vertices, say $k$.
Then $\Gamma_i$ contains at least one vertex $v$ with the following property: 
$v$ is reducible and incident to a collection of contractible subtrees containing no other reducible vertices,
and $v$ becomes reduced after contracting all such subtrees.
Let $T_v$ be one of these contractible subtrees incident to $v$. 

We may rearrange the Neumann moves in $\mathcal{S}$ so that the subtree $T_v$ is created from $v$ and contracted to $v$ by an adjacent subsequence of Neumann moves. This can be done because
$T_v$ contains no other reducible vertices and the other Neumann moves in $\mathcal{S}$ do not involve $T_v\setminus \{v\}$.
We then may excise this subsequence.
Upon rearranging and excising subsequences for each such subtree $T_v$ incident to $v$, we produce a sequence in which the vertex $v$ is reduced in all plumbing trees in which it appears. Applying this procedure for all $\Gamma_i$ with $k$ reducible vertices yields a new sequence $\overline{\mathcal{S}}$ of plumbing trees from $\Gamma$ to $\Gamma'$ for which every non-reduced graph $\overline{\Gamma}_i \in \overline{\mathcal{S}}$ contains at most $k-1$ reducible vertices.
By recursion, the statement follows. 
\end{proof}


\section{Admissible collections}
\label{sec:admser}

Here we define and study admissible collections. These will be used to construct invariant series in \S\ref{sec:qseries}.
We end the section with the proof of Theorem~\ref{thm:uniquenessadm}.

Let $\iota\in W$ be the element defined by
\begin{equation}
\label{eq:iota}
\iota(\alpha) = -\alpha \quad\mbox{for all $\alpha\in Q$}.
\end{equation}
We will apply the Weyl denominator formula
\begin{equation}
\label{eq:Weyldenomformula}
\sum_{w\in W} (-1)^{\ell(w)}\, z^{2w(\rho)} = \prod_{\alpha \in\Delta^+} \left( z^\alpha - z^{-\alpha} \right),
\end{equation}
and a simple consequence of it:

\begin{lemma}
\label{lem:ellw'}
One has $\ell(\iota) \equiv |\Delta^+|$ mod $2$.
Equivalently, one has
\[
(-1)^{\ell(\iota w)} = (-1)^{|\Delta^+|} (-1)^{\ell(w)} \qquad \mbox{for $w\in W$.}
\]
\end{lemma}

\begin{proof}
The statement follows by comparing the coefficients of $z^{2\iota w(\rho)}$ on the two sides of the Weyl denominator formula \eqref{eq:Weyldenomformula}.
\end{proof}

\subsection{Admissible collections}
For a root lattice $Q$, consider a formal series
\begin{equation}
\label{eq:Pz}
P(z) := \sum_{\alpha\in Q} c(\alpha) z^\alpha
\end{equation}
with coefficients $c(\alpha)$ in a commutative ring $R$.
Here $z^\alpha$ for  $\alpha\in Q$ (or more generally, $\alpha$ in the weight lattice) is a multi-index monomial. By linearity, it is enough to assume that $\alpha$ is a root, and in this case, $z^\alpha$ is defined as
\begin{equation}
\label{eq:zpowers}
z^\alpha :=\prod_{i=1}^r z_i^{\langle \alpha^\vee, \lambda_i\rangle}
\end{equation}
with $\alpha^\vee:=\frac{2}{\langle \alpha, \alpha\rangle}\alpha$ being the coroot of $\alpha$ and
$\lambda_1,\dots,\lambda_r$ being the fundamental weights. Hence $P(z)\in R\left\llbracket z_1^{\pm 1}, \dots, z_r^{\pm 1}\right\rrbracket$, the $R$-module of formal series in variables $z_1, \dots, z_r$.

We will consider a collection of such series indexed by elements of $W$ and non-negative integers:
\begin{equation}
\label{eq:P}
\mathcal{P}=\left(P_{x,n}(z) = \sum_{\alpha\in Q} c_{x,n}(\alpha)z^\alpha \in R\left\llbracket z_1^{\pm 1}, \dots, z_r^{\pm 1}\right\rrbracket \,\, \Bigg| \,\, x\in W, \,\, n\geq 0 \right)
\end{equation}
such that for $n=1$, one has
\begin{equation}
\label{eq:Px1}
P_{x,1}(z) = \sum_{w\in W} (-1)^{\ell(w)}\, z^{2 w(\rho)}.
\end{equation}
In particular, $P_{x,1}(z)$ is independent of $x$.
There is a natural action of $W$ on such collections:
\begin{equation}
\label{eq:WonP}
w\left(P_{x,n}(z)\right)_{x,n} = \left(P_{wx,n}(z)\right)_{x,n} \quad \mbox{for $w\in W$}.
\end{equation}
We identify two collections if they are in the same orbit under $W$.

\begin{definition}
A collection of series $\mathcal{P}$ as in \eqref{eq:P} is \textit{admissible} if it satisfies the following four properties:\\[-5pt]

\begin{enumerate}[start=1,label={(P\arabic*)}]

\item \label{it:P1}
One has $P_{x,2}(z) =1$ for all $x\in W$.\\[-5pt]

\item \label{it:P2}
For all $x\in W$ and $n\geq 1$, one has
\begin{align*}
P_{x,n-1}(z) = \left(\sum_{w\in W} (-1)^{\ell(w)}\, z^{2 w(\rho)}\right) P_{x,n}(z).
\end{align*}
Equivalenty,
\[
(-1)^{|\Delta^+|}\sum_{w\in W} (-1)^{\ell(w)}\, c_{x,n}\left( \alpha + 2 w(\rho)\right) = c_{x,n-1}(\alpha)  \qquad \mbox{for $\alpha\in Q$}.
\]
The equivalence of the two conditions is due to Lemma \ref{lem:ellw'} and the fact that
for each $w\in W$, the coefficient of $z^\alpha$ in 
$z^{2\iota w(\rho)} P_{x,n}(z)$
is equal to the coefficient of $z^{\alpha+2w(\rho)}$ in 
$P_{x,n}(z)$ (since $w(\rho)+\iota w(\rho)=0$ in~$Q$).\\[-5pt]

\item \label{it:P3}
For $x\in W$ and $n\geq 0$, the involution $x\mapsto \iota x$ yields
\[
\left[ P_{x,n}(z) \right]_\alpha = (-1)^{|\Delta^+|\,n} \left[ P_{\iota x, n}(z )  \right]_{-\alpha}  \qquad \mbox{for $\alpha\in Q$}.
\]
Equivalently, one has
\begin{equation*}
c_{x,n}(\alpha) = (-1)^{|\Delta^+|\,n} c_{\iota x,n}(-\alpha)  \qquad \mbox{for $\alpha\in Q$}.
\end{equation*}
Here, $\iota \in W$ is as in \eqref{eq:iota}.\\[-5pt]

\item \label{it:P4}
For  $x\in W$ and $p,q\geq 1$, one has 
\[
P_{x,p}(z) P_{x,q}(z) = P_{x,p+q-2}(z).
\]
Equivalently, one has
\begin{equation*}
\sum_{\beta} c_{x,p}\left( \alpha + \beta\right) c_{x,q}\left( -\beta\right) = c_{x, p+q-2}(\alpha) \qquad \mbox{for $\alpha\in Q$}.
\end{equation*}
In particular, the product $P_{x,p}(z) P_{x,q}(z)$ is assumed to exist, hence the sum over $\beta$ here is finite.

\end{enumerate}
\end{definition}

Note that \ref{it:P4} for $p=1$ or $q=1$ together with \eqref{eq:Px1} is equivalent to~\ref{it:P2}.

\begin{definition}
A collection of series $\mathcal{P}$ as in \eqref{eq:P} is \textit{symmetric} if it satisfies the following property:\\[-5pt]

\begin{enumerate}[start=5,label={(P\arabic*)}]

\item \label{it:P5}
For $x,w\in W$ and $n\geq 0$, one has
\[
\left[ P_{x,n}(z) \right]_\alpha = (-1)^{\ell(w)n} \left[ P_{wx,n}(z )  \right]_{w(\alpha)}  \qquad \mbox{for $\alpha\in Q$}.
\]
Equivalently, one has
\begin{equation*}
c_{x,n}(\alpha) = (-1)^{\ell(w)n} c_{wx,n}(w(\alpha))  \qquad \mbox{for $\alpha\in Q$}.
\end{equation*}
\end{enumerate}
\end{definition}

Note that \ref{it:P3} is the special case of \ref{it:P5} with $w=\iota$.

\begin{remark}
\label{rmk:n012}
Properties \ref{it:P1} and \ref{it:P2} with $n\in\{1,2\}$ imply
\[
P_{x,n}(z):=
\left\{
\begin{array}{ll}
\left(\sum\limits_{w\in W} (-1)^{\ell(w)}\, z^{2w(\rho)}\right)^2 & \mbox{if $n=0$,}\\[12pt]
\sum\limits_{w\in W} (-1)^{\ell(w)}\, z^{2w(\rho)} & \mbox{if $n=1$,}\\[12pt]
1 & \mbox{if $n=2$.}
\end{array}
\right.
\]
In particular, $P_{x,n}(z)$ for $n\in\{0,1,2\}$ does not depend on $x$.
Moreover, \ref{it:P4} implies
\[
P_{x,n}(z) = \left(P_{x,3}(z)\right)^{n-2} \quad \mbox{if $n\geq 3$}
\]
for $x\in W$. Thus an admissible collection is determined by a collection of series
\[
P_x(z):= P_{x,3}(z) \quad \mbox{for $x\in W$}
\]
such that $(P_x(z))^n$ exists for $n\geq2$, and 
satisfying \ref{it:P3} and \ref{it:P2} with $n=3$. The latter is:
\begin{align}
\label{eq:P2powern3}
1 = \left(\sum_{w\in W} (-1)^{\ell(\iota w)}\, z^{2\iota w(\rho)}\right) P_x(z).
\end{align}
\end{remark}

\subsection{A key example: the Kostant collection}
\label{sec:Kz}
Consider 
\begin{equation}
\label{eq:Wz}
K(z):= \prod_{\alpha\in \Delta^+}\left(\sum_{i\geq 0} z^{-(2i+1)\alpha} \right).
\end{equation}
Expanding, this is
\[
K(z) = \sum_{\alpha\in Q} k(\alpha)\, z^{-2\rho-2\alpha}
\]
where $k(\alpha)$ is the \textit{Kostant partition function} defined as
\[
k(\alpha):= 
\begin{array}{l}
\mbox{number of ways to represent $\alpha$} \\ 
\mbox{as a sum of positive roots.}
\end{array}
\]
More generally, for $x\in W$, define the \textit{Weyl twist} of $K(z)$ by $x$ as
\begin{equation}
\label{eq:Ktwist}
K_x(z) = (-1)^{\ell(x)}\sum_{\alpha\in Q} k(\alpha)\, z^{-x(2\rho+2\alpha)}.
\end{equation}
Define the \textit{Kostant collection} $\mathcal{K}$ as 
\[
K_{x,n}(z):=
\left\{
\begin{array}{ll}
\left(\sum\limits_{w\in W} (-1)^{\ell(w)}\, z^{2w(\rho)}\right)^2 & \mbox{if $n=0$,}\\[12pt]
\sum\limits_{w\in W} (-1)^{\ell(w)}\, z^{2w(\rho)} & \mbox{if $n=1$,}\\[12pt]
1 & \mbox{if $n=2$,}\\[12pt]
\left(K_x(z)\right)^{n-2} & \mbox{if $n\geq 3$}
\end{array}
\right.
\]
for $x\in W$.
Note that $K_{x,n}(z)$ for $n\in\{0,1,2\}$ does not depend on $x$.

\begin{lemma}
\label{lem:Kadmsym}
The Kostant collection $\mathcal{K}$ is admissible and symmetric.
\end{lemma}

\begin{proof}
After Remark \ref{rmk:n012}, to show that the Kostant collection is admissible 
 it is enough to verify that $K_x(z)$ for $x\in W$ satisfies the following properties: the powers $(K_x(z))^n$ exist for $n\geq 2$; \ref{it:P2} with $n=3$, which is equation \eqref{eq:P2powern3}, holds; 
 and \ref{it:P3} holds.
 Since $K(z)$ is in the ring $\mathbb{Z}\left\llbracket z_1^{-1}, \dots, z_r^{-1}\right\rrbracket$, its positive powers exist. The series $K_x(z)$ is in a similar ring, hence its powers exist for $x\in W$.
Property \eqref{eq:P2powern3} holds from the Weyl denominator formula \eqref{eq:Weyldenomformula} and the fact that
\[
\left( z^\alpha - z^{-\alpha} \right) \left(\sum_{i\geq 0} z^{-(2i+1)\alpha} \right) = 1
\]
for each $\alpha\in \Delta^+$.

For properties \ref{it:P3} and \ref{it:P5}, the case $n\in\{0,1, 2, 3\}$ follows from the definition and the fact that $\ell(wx)\equiv \ell(w)+\ell(x)$ mod $2$. Finally, the case $n\geq 4$ follows since $K_{x,n}(z) = (K_x(z))^{n-2}$.
\end{proof}

\subsection{Proof of uniqueness}
Here we prove Theorem \ref{thm:uniquenessadm}. First we study admissible collections:

\begin{theorem}
\label{thm:adm}
A collection $\mathcal{P}=\left(P_{x,n}(z) \right)_{x,n}$ is admissible  if and only if
\begin{equation}
\label{eq:Padm}
P_{x,n}(z):=
\left\{
\begin{array}{ll}
\left(\sum\limits_{w\in W} (-1)^{\ell(w)}\, z^{2w(\rho)}\right)^2 & \mbox{if $n=0$,}\\[12pt]
\sum\limits_{w\in W} (-1)^{\ell(w)}\, z^{2w(\rho)} & \mbox{if $n=1$,}\\[12pt]
1 & \mbox{if $n=2$,}\\[12pt]
\left(P_x(z)\right)^{n-2} & \mbox{if $n\geq 3$}
\end{array}
\right.
\end{equation}
where 
\[
P_x(z) = K_{f(x)}(z)
\]
for some map of sets $f\colon W\rightarrow W$ such that $f(\iota x)= \iota f(x)$.

In particular, there are only finitely many admissible collections. When $Q=A_1$, 
the collection $\mathcal{P}=\mathcal{K}$ is the unique admissible collection.
\end{theorem}

The map $f$ in the statement induces a map of sets $\overline{f}\colon W/\mathbb{S}_2\rightarrow W/\mathbb{S}_2$, where $\mathbb{S}_2=\{1_W, \iota\}$, which need not be injective, nor surjective. 
In particular,
\begin{equation}
\label{eq:PsetinKset}
\{P_x(z) \,|\, x\in W\} \subseteq \{K_w(z) \,|\, w\in W\}.
\end{equation}
As we identify collections up to the action of $W$ in \eqref{eq:WonP}, the statement implies that an admissible collection is uniquely determined by $\overline{f}$.

\begin{proof}
The backward implication follows as in Lemma \ref{lem:Kadmsym}.
For the forward implication, let $\mathcal{P}$ be an admissible collection. From Remark \ref{rmk:n012},  it suffices to determine $P_x(z):=P_{x,3}(z)$ for $x\in W$.
By \ref{it:P4}, the powers $P^n_x(z)$ exist for $n\geq 2$. However, by definition $P_x(z)$ is in the $R$-module $R\left\llbracket z_1^{\pm 1}, \dots, z_r^{\pm1}\right\rrbracket$, where products in general do not exist.
For the powers $P^n_x(z)$ to exist, one needs that $P_x(z)$ lies in a ring inside $R\left\llbracket z_1^{\pm 1}, \dots, z_r^{\pm1}\right\rrbracket$.
In particular, one needs that $P_x(z)$ lies in a ring of Laurent series where for each Laurent series and each $i=1,\dots,r$, the powers of $z_i$ are either bounded below or bounded above, i.e.,
\[
P_x(z) \in R\left(\!\left(z_1^{\epsilon_1}, \dots, z_r^{\epsilon_r}\right)\!\right)
\]
for some choice of $\epsilon_i\in \{1, -1\}$.
Up to a reflection by a Weyl element, one can assume that for each Laurent series all variables $z_i$ have exponents bounded above, i.e.,
\begin{equation}
\label{eq:PLaurent}
P_x(z) \in R\left(\!\left(z_1^{-1}, \dots, z_r^{-1}\right)\!\right).
\end{equation}
We conclude by showing that $P_x(z)=K(z)$ as in \eqref{eq:Wz}. For this, we use property \ref{it:P2} with $n=3$, which using \ref{it:P1} simplifies as follows: for $P_x(z)=\sum_{\alpha\in Q} c(\alpha)\, z^\alpha$, one has
\begin{equation}
\label{eq:P2n3}
(-1)^{|\Delta^+|}\sum_{w\in W} (-1)^{\ell(w)}\, c\left( \alpha + 2 w(\rho)\right) = 
\left\{
\begin{array}{ll}
1 & \mbox{if $\alpha=0$ in $Q$,}\\[5pt]
0 & \mbox{otherwise.}
\end{array}
\right.
\end{equation}

\begin{figure}[t]
\begin{tikzpicture}
\begin{rootSystem}{A}
\setlength\weightLength{1.5cm}
\weightLattice{6}
\wt [multiplicity=2, black]{1}{0}
\wt [multiplicity=2, black]{-1}{1}
\wt [multiplicity=2, black]{0}{1}
\node[] at \weight{0}{-2} {\footnotesize\(1\)};
\node[] at \weight{0}{-4} {\footnotesize\(2\)};
\node[] at \weight{0}{-6} {\footnotesize\(3\)};
\node[] at \weight{-2}{-2} {\footnotesize\(1\)};
\node[] at \weight{-4}{-2} {\footnotesize\(1\)};
\node[] at \weight{-2}{-4} {\footnotesize\(2\)};
\node[] at \weight{2}{-4} {\footnotesize\(1\)};
\node[] at \weight{4}{-6} {\footnotesize\(1\)};
\node[] at \weight{2}{-6} {\footnotesize\(2\)};
\end{rootSystem}
\end{tikzpicture}
\caption{The lattice points of length at most $12$ in the support of $K(z)$ with the corresponding coefficient in the root lattice $A_2$. Here the three positive roots have length $2$ and are marked with dots.}
\label{fig:A2suppK}
\end{figure}

First, some notation. Define the support of $P(z)$ as
\[
\mathrm{supp}\,P_x(z) :=\{\alpha\in Q \, | \, c(\alpha)\neq 0\}.
\]
One has 
\[
\mathrm{supp}\,K(z) = -2\rho + 2\mathbb{Z}_{\geq 0}\langle\Delta^-\rangle,  
\]
where $\Delta^- = -\Delta^+$ is the set of negative roots (see Figure \ref{fig:A2suppK}). From \eqref{eq:PLaurent}, one has 
\[
\mathrm{supp}\,P_x(z) \supseteq \mathrm{supp}\,K(z).
\]

The argument below uses the lattice points in $\mathrm{supp}\,P_x(z)$ that are larger than the other lattice points in $\mathrm{supp}\,P_x(z)$ --- for two lattice points $\alpha_1$ and $\alpha_2$, one says that $\alpha_1 > \alpha_2$ if $\alpha_1 - \alpha_2\in \mathbb{Z}_{\geq 0}\langle\Delta^+\rangle$, i.e., $\alpha_1 - \alpha_2$ is a sum of positive roots.

\begin{figure}[t]
\begin{tikzpicture}
\begin{rootSystem}{A}
\setlength\weightLength{1.5cm}
\weightLattice{3}
\wt [multiplicity=2, black]{2}{0}
\wt [multiplicity=2, black]{-2}{0}
\wt [multiplicity=2, black]{0}{2}
\wt [multiplicity=2, black]{-2}{2}
\wt [multiplicity=2, black]{2}{-2}
\node[] at \weight{0}{-2} {\footnotesize\(\mu\)};
\node[] at \weight{0}{0} {\footnotesize\(\mu'\)};
\end{rootSystem}
\end{tikzpicture}
\caption{For a lattice point $\mu$ in the root lattice $A_2$, the figure shows $\mu'$ and the lattice points \mbox{$\mu'+2w(\rho)$} with $w\neq \iota$.}
\label{fig:A2mu}
\end{figure}

By contradiction, assume $P_x(z)\neq K(z)$. Write $K(z)=\sum_{\alpha\in Q} d(\alpha)\, z^\alpha$.
Let $\mu \in \mathrm{supp}\,P(z)$ be a maximum lattice point with respect to $>$ such that $c(\mu)\neq d(\mu)$. Such a maximum exists from the assumption \eqref{eq:PLaurent}. 
Apply \eqref{eq:P2n3} with $\mu' :=\mu + 2\rho$, that is:
\begin{equation}
\label{eq:eq:P2n3beta}
(-1)^{|\Delta^+|}\sum_{w\in W} (-1)^{\ell(w)}\, c\left( \mu' + 2 w(\rho)\right) = 
\left\{
\begin{array}{ll}
1 & \mbox{if $\mu'=0$ in $Q$,}\\[5pt]
0 & \mbox{otherwise.}
\end{array}
\right.
\end{equation}
For $w\neq \iota$, one has $\rho + w(\rho)>0$ (see \cite[Corollary on pg.~50]{MR0323842}), thus 
\[
\mu' + 2 w(\rho) > \mu,
\]
hence $c\left(\mu'  + 2 w(\rho)\right)=d\left(\mu' + 2 w(\rho)\right)$ by definition of $\mu$. Then the only summand in \eqref{eq:eq:P2n3beta} which differs from the corresponding identity for the coefficients of $K(z)$ is the one for $w=\iota$. But then \eqref{eq:eq:P2n3beta} implies $c(\mu)=d(\mu)$, a contradiction.
\end{proof}

\begin{proof}[Proof of Theorem \ref{thm:uniquenessadm}]
The first statement follows from Theorem \ref{thm:adm}. For the second one, let $\mathcal{P}$ be an admissible and symmetric collection. From Theorem \ref{thm:adm} and \eqref{eq:PsetinKset}, it is enough to show that
\[
\{P_x(z) \,|\, x\in W\} \supseteq \{K_w(z) \,|\, w\in W\}.
\]
 This follows from \ref{it:P5}, hence the statement.
\end{proof}

\begin{remark}
\label{rmk:admissibleAJK}
It is interesting to compare Theorem \ref{thm:uniquenessadm} with a result from \cite{akhmechet2023lattice}: It is shown there that there are infinitely many collections $\mathcal{P}$ satisfying \eqref{eq:Px1}, \ref{it:P1}, \ref{it:P2}, and \ref{it:P3} --- this result there is stated for $Q=A_1$ but holds similarly for arbitrary $Q$.
Thus Theorem \ref{thm:uniquenessadm} implies that only finitely many of such  collections $\mathcal{P}$ additionally satisfy \ref{it:P4}.
\end{remark}

\subsection{The $A_1$ case}
\label{sec:A1adm}
When $Q=A_1$, the series $K(z)$ from \eqref{eq:Wz} and its Weyl twist $K_{\iota}(z)$ by the action of $\iota\in \mathbb{S}_2$ from \eqref{eq:iota} are
\begin{align*}
K(z) &:= \sum_{i\geq 0} z^{-(2i+1)} & \mbox{and}&&
K_{\iota}(z) & := - \sum_{i\geq 0} z^{2i+1}.
\end{align*}
Since $W=\mathbb{S}_2$, the map $\overline{f}$ from Theorem \ref{thm:adm} is unique, hence the Kostant collection $\mathcal{K}$ is the only admissible collection.

\subsection{The $A_2$ case}
\label{sec:keyA2}

\begin{figure}[t]
\begin{tikzpicture}
\begin{rootSystem}{A}
\setlength\weightLength{1.5cm}
\weightLattice{3}
\node[] at \weight{0}{0} {\footnotesize\(1\)};
\node[] at \weight{0}{1} {\footnotesize\(2\)};
\node[] at \weight{0}{2} {\footnotesize\(3\)};
\node[] at \weight{0}{3} {\footnotesize\(4\)};
\node[] at \weight{1}{0} {\footnotesize\(1\)};
\node[] at \weight{2}{0} {\footnotesize\(1\)};
\node[] at \weight{3}{0} {\footnotesize\(1\)};
\node[] at \weight{1}{1} {\footnotesize\(2\)};
\node[] at \weight{2}{1} {\footnotesize\(2\)};
\node[] at \weight{1}{2} {\footnotesize\(3\)};
\node[] at \weight{-1}{1} {\footnotesize\(1\)};
\node[] at \weight{-1}{2} {\footnotesize\(2\)};
\node[] at \weight{-2}{2} {\footnotesize\(1\)};
\node[] at \weight{-1}{3} {\footnotesize\(3\)};
\node[] at \weight{-2}{3} {\footnotesize\(2\)};
\node[] at \weight{-3}{3} {\footnotesize\(1\)};
\end{rootSystem}
\end{tikzpicture}
\caption{The nonzero coefficients $k(\alpha)$ for the lattice  points of length at most $6$ in the root lattice $A_2$. 
}
\label{fig:A2kalpha}
\end{figure}

When $Q=A_2$, let $\alpha$ and  $\beta$ be the two simple roots. Then $\alpha,\beta$ and $\rho=\alpha+\beta$ are the three positive roots, and the admissible series $K(z)$ from \eqref{eq:Wz} is
\[
K(z) = \left(\sum_{i\geq 0} z^{-(2i+1)\alpha} \right)\left(\sum_{i\geq 0} z^{-(2i+1)\beta} \right)\left(\sum_{i\geq 0} z^{-(2i+1)\rho} \right).
\]
Since $\rho=\alpha+\beta$, a simple computation shows that this expands as
\begin{equation}
\label{eq:WzA2}
K(z) = \sum_{m,n\geq 0} \mathrm{min}\{m,n\}\,  z^{-2m\alpha - 2n\beta}.
\end{equation}
Indeed, the coefficient of $z^{-2m\alpha - 2n\beta}$ here follows from the computation
\begin{align*}
&|\left\{(i, j, k)\in 2\mathbb{N}+1 \,|\, i\alpha + j\beta+k\rho = 2m\alpha + 2n\beta\right\}|\\
&=|\left\{(i, j, k)\in 2\mathbb{N}+1 \,|\, (i+k)\alpha + (j+k)\beta = 2m\alpha + 2n\beta\right\}|\\ 
&=|\left\{k\in 2\mathbb{N}+1 \,|\, k< 2m \mbox{ and } k< 2n\right\}|\\ 
&= \mathrm{min}\{m,n\}.
\end{align*}
These values are represented in Figure \ref{fig:A2kalpha}.
Equivalently, the Kostant partition function from \S\ref{sec:Kz} for $A_2$ is 
\[
k(a\alpha +b\beta)=1+\mathrm{min}\{a,b\} \quad \mbox{for $a,b\geq 0$} 
\]
and vanishes otherwise (see Figure \ref{fig:A2kalpha}). Consequently, one has
\[
K_\iota(z) = -\sum_{m,n\geq 0} \mathrm{min}\{m,n\}\,  z^{2m\alpha +2n\beta}.
\]

Consider the collection $\mathcal{P}$ defined as in \eqref{eq:Padm} by
\begin{align*}
&P_{\alpha}(z) = P_{\beta}(z) = P_{\rho}(z) = K(z), \\
&P_{-\alpha}(z) = P_{-\beta}(z) = P_{-\rho}(z) = K_{\iota}(z).
\end{align*}
By Theorem \ref{thm:adm}, $\mathcal{P}$ is admissible. However $\mathcal{P}$ is not symmetric, as it is not equivalent to the unique symmetric collection $\mathcal{K}$ modulo the action of $W$.


\section{An invariant $q$-series}
\label{sec:qseries}

After defining Weyl assignments for reduced plumbing trees, we define a $q$-series and state the main theorem about its invariance. We conclude with a discussion of the relation with  \cite{gukov2021two, park2020higher, ri2023refined}.

\subsection{Weyl assignments}
\label{sec:chamberass}
For a reduced plumbing tree $\Gamma$ (as in \S\ref{sec:redGamma}) with framing matrix $B$ and a root lattice $Q$, a \textit{Weyl assignment} is a map
\[
\xi\colon V(\Gamma) \rightarrow W, \qquad v\mapsto \xi_v
\]
such that 
\begin{equation}
\label{eq:Wassconditiondeg012}
\xi_v = 1_W \qquad \mbox{if $\deg v \leq 2$},
\end{equation}
with $1_W$ being the identity element in $W$,
and such that the values on vertices across what we call \textit{forcing bridges} are coordinated by the following condition~\eqref{eq:forcingbridgescondition}. 

First, define a \textit{bridge} of $\Gamma$ to be a path in $\Gamma$ 
connecting two vertices, both of degree at least $3$, through a sequence of degree-$2$ vertices. 

Then, define a \textit{forcing bridge} of $\Gamma$ to be a bridge of $\Gamma$ that can be contracted down to a single vertex by a sequence of the Neumann moves  (A$\epsilon$) and (C) from Figure \ref{fig:Neumann}. A forcing bridge of $\Gamma$ between vertices $v$ and $w$ will be denoted by $\Gamma_{v,w}$.

As for contractible branches in \S\ref{sec:redGamma}, forcing bridges 
can be characterized using continued fractions.  
For a bridge $\Gamma_{v,w}$, let $u_1, \dots, u_n$ be the weights of the degree-$2$ vertices between $v$ to $w$ taken in this order. 

\begin{lemma}
A bridge $\Gamma_{v,w}$ is forcing if and only if $\llbracket u_1, \dots, u_n \rrbracket=0$.
\end{lemma}

\begin{proof}
The statement follows by Lemma \ref{lemma:contractiblepath} and the fact that the last Neumann move in the sequence that contracts $\Gamma_{v,w}$ is move (C).
\end{proof}

Finally, for a Weyl assignment $\xi$, one requires
\begin{equation}
\label{eq:forcingbridgescondition}
\xi_v  = \iota^{\Delta\pi(v,w)}  \xi_w \qquad \mbox{for every forcing bridge $\Gamma_{v,w}$}
\end{equation}
where $\iota\in W$ is as in \eqref{eq:iota}, and $\Delta\pi(v,w)$ is defined as the difference in numbers of positive eigenvalues
\begin{equation}
\label{eq:deltapibridge}
\Delta\pi(v,w) := \pi(B)- \pi(\overline{B})
\end{equation}
with $\overline{B}$ being the framing matrix of the plumbing tree obtained from $\Gamma$ after contracting $\Gamma_{v,w}$.

Define
\[
\Xi := \{\mbox{Weyl assignments $\xi$ on $\Gamma$}\}. 
\]
One has
\[
|\Xi| = |W|^n \qquad \mbox{where } n:= |\{v\in V(\Gamma) : \deg v\geq 3\}| - |\{\mbox{forcing bridges}\}|.
\]

\subsection{The $q$-series}
Let $M$ be a weakly negative-definite or weakly positive-definite plumbed $3$-man\-i\-fold. 
After a sequence of Neumann moves, one can assume that $M$ is constructed from a plumbing tree $\Gamma$ which is 
reduced and either weakly negative definite or weakly positive definite (see Lemma \ref{lemma:weakcontract}). 
For a root lattice $Q$,
select a representative $a$ of a generalized $\mathrm{Spin}^c$-structure 
\begin{equation}
\label{eq:a}
a\in \delta +2L' \otimes_{\mathbb{Z}} Q \subset L' \otimes_{\mathbb{Z}} Q,
\end{equation} 
a collection $\mathcal{P}$ of series as in \eqref{eq:P} with coefficients in a commutative ring $R$, and a subset $S\subseteq W^{V(\Gamma)}$. 
Define 
\begin{equation}
\label{eq:Ypa}
\mathsf{Y}_{\mathcal{P},S,a}\left(q\right):= 
(-1)^{|\Delta^+| \,\pi}
q^{\frac{1}{2}(3\sigma-\mathrm{tr}\,B)\langle \rho, \rho \rangle} \frac{1}{|S|}\sum_{\xi\in S}\,\sum_{\ell\in a + 2BL\otimes Q}
c_{\Gamma,\xi} (\ell)\,
q^{-\frac{1}{8}\langle \ell, \ell\rangle}
\end{equation}
where
\begin{align}
\label{eq:cGammaxiell}
c_{\Gamma, \xi} (\ell)  := \prod_{v\in V(\Gamma)} \left[ P_{\xi_v, \deg v}(z_v) \right]_{\ell_v} \in R.
\end{align}
Here $P_{x,n}$ denotes the element of $\mathcal{P}$ as in \S\ref{sec:admser}, the operator  $[\,\,]_\alpha$ assigns to a series in $z$ the coefficient of the monomial $z^\alpha$ for $\alpha\in Q$, and $\ell_v\in Q$ denotes the $v$-component of $\ell\in L'\otimes_{\mathbb{Z}} Q\cong Q^{V(\Gamma)}$ for $v\in V(\Gamma)$.

The series $\mathsf{Y}_{\mathcal{P},S,a}\left(q\right)$ naturally decomposes as
\[
\mathsf{Y}_{\mathcal{P},S,a}\left(q\right) = \frac{1}{|S|}\sum_{\xi\in S} \mathsf{Y}_{\mathcal{P},\xi,a}\left(q\right)
\]
where $\mathsf{Y}_{\mathcal{P},\xi,a}\left(q\right)$ is the series for the subset $\{\xi\}\subset W^{V(\Gamma)}$.

\begin{remark}
\label{rmk:YPa}
The sum in the series is over $\ell$ such that $\ell\equiv a$ mod $2BL\otimes Q$. As these $\ell$ are all the representative of the class of $a$ in the quotient space  $\mathsf{B}_Q(M)$ from \eqref{eq:BQM}, the series $\mathsf{Y}_{P,a}(q)$  depends on  $a$ at most up to its class in $\mathsf{B}_Q(M)$. 
Moreover, we show in Theorem \ref{thm:Winvqseries} that  $\mathsf{Y}_{\mathcal{P},S,a}(q)$ for $S=\Xi$ depends on  $a$ only up to its class in the  space $\mathsf{B}_Q^{W}(M)$, which is the quotient of $\mathsf{B}_Q(M)$ modulo the action of the Weyl group $W$.
\end{remark}

\begin{lemma}
\label{lemma:Laurent}
If the plumbing tree is  weakly negative definite, then the powers of $q$ in $\mathsf{Y}\left(q\right)$
 are bounded below, and for each power of $q$ there are only finitely many contributions to $\mathsf{Y}\left(q\right)$. 
In particular, one has
\begin{equation}
\label{eq:YRing}
\mathsf{Y}\left(q\right) 
\quad\in\quad q^{\frac{1}{2}(3\sigma-\mathrm{tr}\,B)\langle \rho, \rho \rangle -\frac{1}{8} \langle a, a \rangle} \frac{1}{|S|}\,R
\left(\! \left(q^{\frac{1}{2}}\right)\!\right).
\end{equation}
Similarly, if the plumbing tree is  weakly positive definite, then 
\begin{equation*}
\mathsf{Y}\left(q\right) 
\quad\in\quad q^{\frac{1}{2}(3\sigma-\mathrm{tr}\,B)\langle \rho, \rho \rangle -\frac{1}{8} \langle a, a \rangle} \frac{1}{|S|}\,R
\left(\! \left(q^{-\frac{1}{2}}\right)\!\right).
\end{equation*}
\end{lemma}

\begin{proof}
This is similar to the argument for the series $\widehat{Z}_a(q)$ from \cite{gukov2021two}.
The sum over $\ell$ can be decomposed as a sum over the entries of $\ell$ corresponding to vertices of degree at most $2$
 and the entries corresponding to vertices of degree at least $3$. For the former ones, there are only finitely many contributions due to the definition of admissible collections, see Remark \ref{rmk:n012}. For the latter ones, the boundedness of the exponents of $q$ and the finiteness of the contributions to each power of $q$ follow from the assumption that the plumbing tree is weakly negative definite. Recall from \S\ref{sec:ref} that this implies that the inverse $B^{-1}$ of the framing matrix is negative definite on the subspace spanned by the vertices of degree at least $3$.

Finally, the exponents of $q$ in \eqref{eq:YRing} follow from the fact that for an element $\ell\in a + 2BL\otimes Q$, one has $\langle \ell, \ell \rangle \in \langle a, a \rangle + 4\mathbb{Z}$.
\end{proof}

The Laurent ring in \eqref{eq:YRing} can be simplified for $Q=A_1$: since the pairing in $A_1$ is always even, one has $\langle \ell, \ell \rangle \in \langle a, a \rangle + 8\mathbb{Z}$, hence \eqref{eq:YRing} for $Q=A_1$ becomes
\[
\mathsf{Y}\left(q\right) \quad\in\quad q^{\frac{1}{4}(3\sigma-\mathrm{tr}\,B - a^2)} \frac{1}{|S|}\,R
\left(\! \left(q\right)\!\right).
\]
Here we use that $\langle \rho, \rho \rangle=\frac{1}{2}$ for $Q=A_1$.

\smallskip

Our  main result here is:

\begin{theorem}
\label{thm:seriesinvariancef}
For an admissible $\mathcal{P}$ and $S\subseteq\Xi$, 
any two reduced plumbing trees for $M$ which are related by a series of the five Neumann moves from Figure \ref{fig:Neumann}
yield the same series $\mathsf{Y}_{\mathcal{P},S,a}\left(q\right)$.
\end{theorem}

Here $\Xi$ is the set of Weyl assignments from \S\ref{sec:chamberass}.
We prove the statement in \S\ref{sec:proofinv}.
Combining Lemma \ref{lemma:Laurent} and Theorem \ref{thm:seriesinvariancef}, we immediately deduce:

\begin{corollary}
For an admissible $\mathcal{P}$, a set $S\subseteq\Xi$, and 
 a reduced and refinable plumbing tree $\Gamma$, the series $\mathsf{Y}_{\mathcal{P},S,a}\left(q\right)$ is a Laurent series in either $q^{1/2}$ or $q^{-1/2}$ with integer coefficients, up to an overall factor that is a rational power of $q$.
\end{corollary}

\begin{remark}
\label{rmk:P=K}
Given a collection $\mathcal{P}$ that is admissible, the resulting series $\mathsf{Y}_{\mathcal{P},S,a}\left(q\right)$ may be obtained from the Kostant collection $\mathcal{K}$ from \S\ref{sec:Kz} together with a possibly different set $S' \subseteq \Xi$. Indeed, from Theorem \ref{thm:adm}, an admissible collection is uniquely determined by
\[
P_x(z) = K_{f(x)}(z)
\]
for some map of sets $f\colon W\rightarrow W$ such that $f(\iota x)= \iota f(x)$. It follows from its definition that the series $\mathsf{Y}_{\mathcal{P},S,a}\left(q\right)$ coincides with the series obtained for $\mathcal{P}=\mathcal{K}$ and some $S'\subseteq \Xi\cap \mathrm{Image}(f)^{V(\Gamma)}$. Moreover, in this case one has by definition
\begin{equation}
\label{eq:Wxwa}
\mathsf{Y}_{\mathcal{K}, \xi,a}(q) = \mathsf{Y}_{\mathcal{K}, w(\xi), w(a)}(q) \qquad \mbox{for $\xi\in \Xi$ and $w\in W$}.
\end{equation}
\end{remark}

\subsection{The $A_1$ case and Ri's series}
For $Q=A_1$, $\mathcal{P}=\mathcal{K}$ the Kostant collection, and $S=\Xi$,  the resulting series $\mathsf{Y}_{\mathcal{P},S,a}\left(q\right)$ coincides with the $q$-series from \cite{ri2023refined}.

\subsection{Relation with the series $\widehat{Z}_a(q)$}
\label{sec:Park}
Assume that $S=\Xi$ and the reduced weakly negative-definite plumbing tree $\Gamma$ has no forcing bridges (e.g., this is the case when $\Gamma$ has at most one vertex of degree at least $3$). Then for each $\ell$, the coefficients $c_{\Gamma,\xi} (\ell)$ from \eqref{eq:cGammaxiell} satisfy
\begin{equation}
\label{eq:c_gammanoforce}
\frac{1}{|\Xi|}\sum_{\xi\in \Xi}c_{\Gamma,\xi} (\ell) =  \prod_{v\in V(\Gamma)}   \left( \frac{1}{|W|} \sum_{x\in W}  \left[ P_{x,\deg v}(z_v) \right]_{\ell_v} \right).
\end{equation}
This is due to the fact that in the absence of forcing bridges, the values of the Weyl assignments on various vertices do not need to be coordinated as in \eqref{eq:forcingbridgescondition}; and the fact that for $n\in\{0,1,2\}$, the series $P_{x,n}(z)$ is independent of~$x$, and thus
\begin{equation}
\label{eq:aven012}
P_{x,n}(z)=\frac{1}{|W|}\sum_{x\in W}P_{x,n}(z) \qquad \mbox{for $n\in\{0,1,2\}$}.
\end{equation}

Now, consider the case $\mathcal{P}=\mathcal{K}$ as in \S\ref{sec:Kz}.

\begin{lemma}
For arbitrary $n\geq 0$ and $\alpha\in Q$, one has
\begin{align*}
\begin{split}
&\frac{1}{|W|}\left[\sum_{x\in W} K_{x,n}(z) \right]_\alpha\\
&= \mathsf{v.p.} \oint_{|z_1|=1} \dots  \mathsf{v.p.} \oint_{|z_r|=1} 
\left( \sum_{w\in W} (-1)^{\ell(w)}\, z^{2w(\rho)} \right)^{2-n}
 z^{-\alpha} \,\prod_{k=1}^r  \frac{dz_k}{2\pi \mathrm{i} z_k}.
\end{split}
\end{align*}
\end{lemma}

\begin{proof}
Here $\mathsf{v.p.}$ stands for the principal value (\textit{valeur principale} in French) of the integral and is computed as follows. For $n\in\{0,1,2\}$, the term
\begin{equation}
\label{eq:integrandmono}
\left( \sum_{w\in W} (-1)^{\ell(w)}\, z^{2w(\rho)} \right)^{2-n}
\end{equation}
expands as a finite sum, and the $\mathsf{v.p.}$ integral is simply the regular integral. 
Hence, for $\alpha\in Q$, the right-hand side 
equals the coefficient of $z^\alpha$ in the expansion of \eqref{eq:integrandmono}, and 
the lemma holds by \eqref{eq:aven012} and the definition of $\mathcal{K}$. 

For $n\geq 3$, applying the Weyl denominator formula \eqref{eq:Weyldenomformula}, one sees that the term \eqref{eq:integrandmono} is singular when $z^\beta=\pm 1$ for $\beta \in \Delta^+$. In this case, the term \eqref{eq:integrandmono} admits various series expansions, one in each Weyl chamber. For $\alpha\in Q$, the right-hand side 
is defined as the average of the coefficients of $z^\alpha$ among these various series expansions. As the various series expansions coincide with $K_{x,n}(z)$ for $x\in W$, the lemma holds.
\end{proof}

It follows that \eqref{eq:c_gammanoforce} can be expressed as a product of v.p.~integrals. In particular, when the plumbing tree $\Gamma$ is reduced weakly negative definite  and has no forcing bridges, the series $\mathsf{Y}_{W,a}\left(q\right)$ recovers the series $\widehat{Z}_a(q)$ from \cite{park2020higher}. And for $Q=A_1$, this equals the series $\widehat{Z}_a(q)$ from \cite{gukov2020bps, gukov2021two}.


\section{Invariance of the $q$-series}
\label{sec:proofinv}

Here we consider the series $\mathsf{Y}\left(q\right):=\mathsf{Y}_{\mathcal{P},S,a}\left(q\right)$ with $\mathcal{P}$ admissible as in \S\ref{sec:admser} and $S=W^{V(\Gamma)}$.
We first prove the invariance of the series $\mathsf{Y}\left(q\right)$ with respect to the action of the Weyl group and then prove Theorem \ref{thm:seriesinvariancef}.

Recall the coefficients $c_{\Gamma, \xi}(\ell)$ from \eqref{eq:cGammaxiell}.

\begin{theorem}
\label{thm:Winvqseries}
If $\mathcal{P}$ is symmetric and $S=\Xi$,
then for $\ell\in \delta +2L' \otimes_{\mathbb{Z}} Q \subset L' \otimes_{\mathbb{Z}} Q$, one has
\[
\sum_{\xi\in \Xi}c_{\Gamma,\xi}(\ell)=\sum_{\xi\in \Xi} c_{\Gamma,\xi}(w(\ell)) \qquad \mbox{for  $w\in W$}.
\]
In particular, when it exists, the series $\mathsf{Y}\left(q\right)$ is invariant by the action of the Weyl group $W$, that is, $\mathsf{Y}_{\mathcal{P},\Xi,a}\left(q\right)=\mathsf{Y}_{\mathcal{P},\Xi,w(a)}\left(q\right)$,  for $w\in W$.
\end{theorem}

\begin{proof}
The first part of the statement follows after multiplying the identity in property \ref{it:P5} over all vertices in the reduced plumbing tree $\Gamma$, summing over all $\xi\in \Xi$, and applying the identity 
\[
\prod_{v\in V(\Gamma)} (-1)^{\deg v} =1.
\]
This identity follows since $\sum_{v\in V(\Gamma)} \deg v$ is even, as every edge of $\Gamma$ is incident to two vertices in $\Gamma$.
Moreover, the last part of the statement follows from the fact that $\langle \ell, \ell\rangle=\langle w(\ell), w(\ell)\rangle$ for $w\in W$, hence the exponent of $q$ is also invariant by the action of $W$.
\end{proof}

We now prove Theorem \ref{thm:seriesinvariancef}:

\begin{proof}[Proof of Theorem \ref{thm:seriesinvariancef}]
Let $M$ be a weakly negative-definite plumbed $3$-man\-i\-fold, and let $Q$ be a root lattice. Select an admissible collection $\mathcal{P}$ for~$Q$.
We verify that any two reduced plumbing trees for $M$ which are related by a series of the five Neumann moves from Figure \ref{fig:Neumann}
yield the same series $\mathsf{Y}_{a}\left(q\right)$ for all representatives $a$ of a generalized $\mathrm{Spin}^c$-structure for $Q$. We use the fact that any two reduced plumbing trees for $M$ which are related by a series of the five Neumann moves from Figure \ref{fig:Neumann}
 are in fact related by a sequence of the Neumann moves from Figure \ref{fig:Neumann}  which do not create any reducible vertices \cite[Prop.~3.4]{ri2023refined}.

For each such move, we argue that the two $q$-series arising from the two plumbing trees are equal.
As in the proof of Proposition~\ref{prop:BQMinvariance}, we use the notation $B\colon L\hookrightarrow L'$ and $\delta$ for the terms related to the bottom plumbing tree $\Gamma$, and the notation
\mbox{$B_\circ \colon L_\circ \hookrightarrow L_\circ'$} and $\delta_\circ$ for the corresponding terms related to the top plumbing  tree $\Gamma_\circ$.
The signatures of $B$ and $B_\circ$ will be denoted by $\sigma$ and $\sigma_\circ$, respectively, and the numbers of positive eigenvalues of $B$ and $B_\circ$ will be denoted by $\pi$ and $\pi_\circ$, respectively. 

Select a representative of a generalized $\mathrm{Spin}^c$-structure $a$ for the bottom plumbing tree.
For each move, we start by observing how the factor 
\begin{equation}
\label{eq:frontfactor}
(-1)^{|\Delta^+| \,\pi}
q^{\frac{1}{2}(3\sigma-\mathrm{tr}\,B)\langle \rho, \rho \rangle} 
\end{equation}
in front of the sum in the series changes under the move. Afterwards, we focus on the sum over the various representatives $\ell\in a + 2BL\otimes Q$ of the generalized $\mathrm{Spin}^c$-structure. For this, recall from Proposition \ref{prop:BQMinvariance2} that the space of generalized $\mathrm{Spin}^c$-structures is invariant under the Neumann moves. However, for each move, the column space of $B$ is isomorphic to some subspace of the column space of $B_\circ$, see Remark \ref{rmk:columnspace}. It follows that for each representative $\ell$ of the generalized $\mathrm{Spin}^c$-structure for the bottom tree, there is a corresponding affine space of generalized $\mathrm{Spin}^c$-structures for the top plumbing tree. Thus for each move, we argue that the contribution of each $\ell$ for the bottom plumbing tree equals the sum of the contributions of the elements in the corresponding affine space for the top plumbing tree. 

Moreover, for each move, we use the assumption that $\Gamma$ and $\Gamma_\circ$ are reduced to show that their sets of Weyl assignments are isomorphic. Then as the series for $S\subseteq \Xi$ is the average of the series for the $1$-element subsets of $S$, it is enough to prove the statement when $S$ has size $1$.

\bigskip

\noindent \textit{Step (A$-$): The Neumann move (A$-$) from Figure \ref{fig:Neumann}.}
There exists an extra term in the quadratic form corresponding to $B_\circ$ with respect to the quadratic form corresponding to $B$ given by
\[
-x_0^2 - x_1^2 - x_2^2 +2x_0x_1 + 2x_0x_2 -2x_1x_2=
-(x_0-x_1-x_2)^2,
\]
where $x_0$ is the variable corresponding to the added vertex and $x_1$ and $x_2$ are the variables corresponding to its two adjacent vertices in $\Gamma_\circ$. It follows that 
\[
\sigma_\circ=\sigma-1 \qquad \mbox{and} \qquad \pi_\circ = \pi. 
\]
Since $\mathrm{tr}\, B_\circ = \mathrm{tr}\, B -3$, one has $3\sigma_\circ - \mathrm{tr}\, B_\circ = 3\sigma - \mathrm{tr}\, B$. We conclude that the factor \eqref{eq:frontfactor} in front of the sum in the series is invariant under this move.

Next, we consider the sum in the series.
Recall the function $R$ from \eqref{eq:Ra} with $\epsilon = -$:
\begin{equation*}
R\colon L'\otimes_{\mathbb{Z}} Q \rightarrow L'_\circ \otimes_{\mathbb{Z}} Q, \qquad
(a_1, a_2)\mapsto (a_1, 0,  a_2).
\end{equation*}
Here the subtuple $a_1$ corresponds to the vertices of $\Gamma$ consisting of the vertex labeled by $m_1$ and all vertices on its left.  The subtuple $a_2$ corresponds to the vertices of $\Gamma$ consisting of the vertex labeled by $m_2$ and all vertices on its right. The $0$ entry corresponds to the added vertex in $\Gamma_\circ$.
(See Remark \ref{rmk:leftright} about the determination of left and right parts of the trees.)
For \mbox{$a\in \delta +2L' \otimes_{\mathbb{Z}} Q$,} define 
$a_\circ:= R(a)\in \delta_\circ +2L'_\circ \otimes_{\mathbb{Z}} Q$.

Since $\Gamma_\circ$ does not have a new vertex  of degree $\geq 3$, nor has it a new forcing bridge, the top and bottom plumbing trees have isomorphic sets of Weyl assignments. For a Weyl assignment $\xi$ on $\Gamma$, let $\xi_\circ$ be the naturally induced Weyl assignment on $\Gamma_\circ$.

The added vertex in $\Gamma_\circ$ has degree $2$. From \ref{it:P1}, we deduce that for $\ell_\circ\in a_\circ + 2B_\circ L_\circ\otimes_{\mathbb{Z}} Q$,
one has $c_{\Gamma_\circ, \xi_\circ}(\ell_\circ)=0$ when the component of $\ell_\circ$ corresponding to the added vertex is non-zero. Hence, we can restrict the sum in the series for $\Gamma_\circ$ over only those $\ell_\circ$ which are of type $\ell_\circ= R(\ell)$ for some $\ell\in a + 2B L\otimes_{\mathbb{Z}} Q$. 
As $R$ is injective, it will be enough to verify that the contribution of $\ell\in a + 2B L\otimes_{\mathbb{Z}} Q$ in the series for $\Gamma$ equals the contribution of $R(\ell)$ in the series for $\Gamma_\circ$.

From \ref{it:P1}, one has $c_{\Gamma_\circ, \xi_\circ}(R(\ell))=c_{\Gamma, \xi}(\ell)$.
Moreover, a direct computation shows that 
\[
B^{-1}
\ell
 =
\left(h_1,  h_2 \right)
\quad \Rightarrow \quad 
B_\circ^{-1}
R(\ell)
=
\left(h_1, h_0, h_2\right)
\]
for some $h_0$. (Specifically, $h_0$ is the sum of the entry of $h_1$ and the entry of $h_2$ corresponding to the two vertices adjacent to the added vertex in $\Gamma_\circ$. However, the explicit expression of $h_0$ will not be needed below.)
This implies that writing $\ell=(\ell_1, \ell_2)$, one has
\begin{equation}
\label{eq:RRllA-}
\langle R(\ell), R(\ell) \rangle = (\ell_1, 0, \ell_2)^t (h_1, h_0, h_2) =\langle \ell, \ell \rangle.
\end{equation}
 
We conclude that 
\begin{equation}
\label{eq:conteqA-}
c_{\Gamma, \xi}(\ell) \,
q^{-\frac{1}{8}\langle \ell, \ell\rangle} =
c_{\Gamma_\circ, \xi_\circ}(R(\ell))\,
q^{-\frac{1}{8}\langle R(\ell), R(\ell)\rangle}.
\end{equation}
Hence the contribution of $\ell$ in the series for $\Gamma$ equals the contribution of $R(\ell)$ in the series for $\Gamma_\circ$.
This implies the statement for this move.

\bigskip

\noindent \textit{Step (A$+$): The Neumann move (A$+$) from Figure \ref{fig:Neumann}.}
 In this case, one has
\[
\sigma_\circ=\sigma +1, \qquad \pi_\circ = 1+ \pi , \qquad 3\sigma_\circ - \mathrm{tr}\, B_\circ =  3\sigma - \mathrm{tr}\, B. 
\]
We conclude that the factor \eqref{eq:frontfactor} in front of the sum in the series for $\Gamma_\circ$ has an extra factor $(-1)^{|\Delta^+|}$.

Next, we use the function $R$ from \eqref{eq:Ra} this time with $\epsilon = +$:
\begin{equation*}
R\colon L'\otimes_{\mathbb{Z}} Q \rightarrow L'_\circ \otimes_{\mathbb{Z}} Q, \qquad
(a_1, a_2)\mapsto (a_1, 0,  -a_2).
\end{equation*}
For \mbox{$a\in \delta +2L' \otimes_{\mathbb{Z}} Q$,} define 
$a_\circ:= R(a)\in \delta_\circ +2L'_\circ \otimes_{\mathbb{Z}} Q$.

As for the previous move, the sets of Weyl assignments for the two plumbing trees are isomorphic. The natural isomorphism is defined as follows. For a Weyl assignment  $\xi$ for the bottom plumbing tree,  define a Weyl assignment $\xi_\circ$ for the top plumbing tree such that for a vertex $v$ with $\deg v\geq 3$, one has
\begin{equation}
\label{eq:identifXiA+}
\xi_\circ \colon v \mapsto \left\{
\begin{array}{ll}
\xi_v & \mbox{if $v$ is on the left of the added vertex,}\\[0.5pt]
\iota \xi_v & \mbox{if $v$ is on the right of the added vertex.} 
\end{array}
\right.
\end{equation}
Here, $\iota\in W$ is as in \eqref{eq:iota}.
Since the added vertex has degree $2$, the value of $\xi_\circ$ at the added vertex is $1_W$, as determined by \eqref{eq:Wassconditiondeg012}.

Note that when the added vertex is on a forcing bridge $\Gamma_{v,w}$, the definition of $\xi_\circ$ via \eqref{eq:identifXiA+} is compatible with the condition \eqref{eq:forcingbridgescondition}, since 
\[
\Delta\pi_\circ(v,w) = \Delta\pi(v,w) +1
\]
where $\Delta\pi(v,w)$ and $\Delta\pi_\circ(v,w)$ are the differences in numbers of positive eigenvalues obtained from the contraction of the bridge $\Gamma_{v,w}$  in $\Gamma$ and $\Gamma_\circ$, respectively, as in \eqref{eq:deltapibridge}.

As for the previous move, we can restrict the sum in the series for $\Gamma_\circ$ over only those $\ell_\circ$ which are of type $\ell_\circ= R(\ell)$ for some $\ell\in a + 2B L\otimes_{\mathbb{Z}} Q$. In this case, we have 
\begin{equation}
\label{eq:fgammaB+}
c_{\Gamma, \xi}(\ell) = (-1)^{|\Delta^+|} c_{\Gamma_\circ, \xi_\circ}(R(\ell)).
\end{equation}
This follows from \ref{it:P3}, the definition of $\xi_\circ$, and the fact that 
\begin{equation}
\label{eq:oddsumdegv}
\prod_{v\in V_2(\Gamma_\circ)}(-1)^{\deg v}=-1 
\end{equation}
where $V_2(\Gamma_\circ)$ is the set of all vertices of $\Gamma_\circ$ on the right of the added vertex.
Indeed, one has that $\sum_{v\in V_2(\Gamma_\circ)} \deg v$ is odd, since every edge on the right of the added vertex in $\Gamma_\circ$ is incident to two vertices in $V_2(\Gamma_\circ)$ with the exception of the edge incident to the added vertex, which is incident to only one vertex in $V_2(\Gamma_\circ)$.
The factor $(-1)^{|\Delta^+|}$ in \eqref{eq:fgammaB+} matches the extra contribution to the factor \eqref{eq:frontfactor} in front of the sum in the series for $\Gamma_\circ$. That is, we have
\begin{equation*}
(-1)^{|\Delta^+|\, \pi} c_{\Gamma, \xi}(\ell) = (-1)^{|\Delta^+|\, \pi_\circ} c_{\Gamma_\circ, \xi_\circ}(R(\ell)).
\end{equation*}

A direct computation shows that 
\[
B^{-1}
\ell
 =
\left(h_1,  h_2 \right)
\quad \Rightarrow \quad 
B_\circ^{-1}
R(\ell)
=
\left(h_1, h_0, -h_2\right)
\]
for some $h_0$. (Specifically, $h_0$ is minus the sum of the entry of $h_1$ and the entry of $-h_2$ corresponding to the two vertices adjacent to the added vertex in $\Gamma_\circ$; however, the formula for $h_0$ will not be needed below.)
This implies that 
\begin{equation}
\label{eq:RRllA+}
\langle R(\ell), R(\ell) \rangle = (\ell_1, 0, -\ell_2)^t (h_1, h_0, - h_2) =\langle \ell, \ell \rangle.
\end{equation}
 
We conclude that
\begin{equation}
\label{eq:conteqA+}
(-1)^{|\Delta^+|\, \pi} c_{\Gamma, \xi}(\ell) \,
q^{-\frac{1}{8}\langle \ell, \ell\rangle} =
(-1)^{|\Delta^+|\, \pi_\circ} c_{\Gamma_\circ, \xi_\circ}(R(\ell))\,
q^{-\frac{1}{8}\langle R(\ell), R(\ell)\rangle}.
\end{equation}
Hence the contribution of $\ell$ in the series for $\Gamma$ equals the contribution of $R(\ell)$ in the series for $\Gamma_\circ$.
Since $R$ is injective, the statement for this move follows.

\bigskip

\noindent \textit{Step (B$-$): The Neumann move (B$-$) from Figure \ref{fig:Neumann}.} 
In this case, one has
\[
\sigma_\circ=\sigma-1, \qquad \pi_\circ = \pi, \qquad 3\sigma_\circ - \mathrm{tr}\, B_\circ = -1+ 3\sigma - \mathrm{tr}\, B. 
\]
We conclude that the factor \eqref{eq:frontfactor} in front of the sum in the series for $\Gamma_\circ$ has an extra factor $q^{-\frac{1}{2}\langle \rho, \rho\rangle}$.

For a choice of $w\in W$, consider the function:
\begin{equation}
\label{eq:RwB-}
R_w\colon L'\otimes_{\mathbb{Z}} Q \rightarrow L'_\circ \otimes_{\mathbb{Z}} Q, \qquad
(a_\sharp, a_1)\mapsto (a_\sharp, a_1+2w(\rho),  - 2 w(\rho))
\end{equation}
with entry $a_1$ corresponding to the vertex of $\Gamma$ labeled by $m_1$, subtuple $a_\sharp$ corresponding to all other vertices of $\Gamma$,
and entry $- 2 w(\rho)$ on the right-hand side corresponding to the added vertex in $\Gamma_\circ$.
Note that the function $R$ from \eqref{eq:RaB} with $\epsilon = -$ is $R=R_{w}$ with $w=1_W$.
For $a\in \delta +2L' \otimes_{\mathbb{Z}} Q$, define 
$a_\circ:= R(a)\in \delta_\circ +2L'_\circ \otimes_{\mathbb{Z}} Q$.

The added vertex in $\Gamma_\circ$ has degree $1$. From \eqref{eq:Px1}, 
we deduce that for $\ell_\circ\in a_\circ + 2B_\circ L_\circ\otimes_{\mathbb{Z}} Q$,
one has $c_{\Gamma_\circ, \xi_\circ}(l_\circ)=0$ when the component of $\ell_\circ$ corresponding to the added vertex is not in the orbit $-2W(\rho)$. Hence, we can restrict the sum in the series for $\Gamma_\circ$ over only those $\ell_\circ$ which are of type $\ell_\circ= R_w(\ell)$ for some $\ell\in a + 2B L\otimes_{\mathbb{Z}} Q$ and some $w\in W$. 
Note that for $\ell\in a + 2B L\otimes_{\mathbb{Z}} Q$ and $w\in W$, one indeed has
\[
R_w(\ell) \in a_\circ + 2B_\circ L_\circ\otimes_{\mathbb{Z}} Q.
\]

Let $n$ be the degree of the vertex in $\Gamma_\circ$ adjacent to the added vertex. Then the degree of the corresponding vertex in $\Gamma$ is $n-1$. 
The assumption that this Neumann move does not create a reducible vertex implies $n\neq 3$.
It follows that $\Gamma_\circ$ does not have a new vertex of degree at least $3$, nor has it a new forcing bridge, 
hence the top and bottom plumbing trees have isomorphic sets of Weyl assignments.

For $\ell\in a + 2B L\otimes_{\mathbb{Z}} Q$, write $\ell=(\ell_\sharp, \ell_1)$.
Select a Weyl assignment $\xi$ on $\Gamma$, let $\xi_\circ$ be the corresponding Weyl assignment on $\Gamma_\circ$, and let $x\in W$ be the value of $\xi_\circ$ at the vertex in $\Gamma_\circ$ adjacent to the added vertex.
From  \ref{it:P2}, one has
\begin{align}
\label{eq:moveB-Lemma2app}
\begin{split}
c_{x,n-1}(\ell_1) &= 
(-1)^{|\Delta^+|}\sum_{w\in W} (-1)^{\ell(w)}\, c_{x,n}\left( \ell_1 + 2 w(\rho)\right) \\
&= \sum_{w\in W} c_{1_W,1}\left( -2w(\rho)\right) c_{x,n}\left( \ell_1 + 2 w(\rho)\right).
\end{split}
\end{align}
The second identity follows from Lemma \ref{lem:ellw'} and \eqref{eq:Px1}, 
 which together imply 
\[
(-1)^{|\Delta^+| } (-1)^{\ell(w)} = (-1)^{\ell(\iota w)} = c_{1_W,1}\left( -2w(\rho)\right) \qquad \mbox{for $w\in W$.}
\]
Multiplying both sides of \eqref{eq:moveB-Lemma2app} by the contributions corresponding to the remaining vertices of $\Gamma$, one has
\begin{equation*}
\prod_{v\in V(\Gamma)} \left[ P_{\xi_v, \deg v}(z_v) \right]_{\ell_v}
= 
\sum_{w\in W} 
\prod_{v\in V(\Gamma_\circ)} \left[ P_{\xi_\circ(v), \deg v}(z_v) \right]_{R_w(\ell)_v}.
\end{equation*}
I.e., one has
\begin{equation}
\label{eq:f=sumf}
c_{\Gamma, \xi}(\ell) = \sum_{w\in W}c_{\Gamma_\circ, \xi_\circ}(R_w(\ell)).
\end{equation}

Next, we consider the powers of $q$. For $w\in W$, a direct computation shows that 
\[
B^{-1}
\ell
 =
\left(h_\sharp,  h_1\right)
\quad \Rightarrow \quad 
B_\circ^{-1}
R_w(\ell)
=
\left(h_\sharp, h_1,h_1+2w(\rho) \right).
\]
This implies that 
\begin{align*}
\langle R_w(\ell), R_w(\ell) \rangle &= 
\left( \ell_\sharp , \ell_1+2w(\rho), -2w(\rho) \right)^t
\left(h_\sharp, h_1, h_1+2w(\rho)\right)\\
&=\langle \ell, \ell \rangle -4 \langle \rho, \rho \rangle.
\end{align*}
Thus from \eqref{eq:f=sumf}, we have
\begin{equation}
\label{eq:eqcontB-}
c_{\Gamma, \xi}(\ell) \, q^{-\frac{1}{8}\langle \ell, \ell\rangle}
= 
q^{-\frac{1}{2}\langle \rho,\rho\rangle}
\sum_{w\in W}c_{\Gamma_\circ, \xi_\circ}(R_w(\ell))\,
q^{-\frac{1}{8}\langle R_w(\ell), R_w(\ell)\rangle}.
\end{equation}
The factor $q^{-\frac{1}{2}\langle \rho,\rho\rangle}$ on the right-hand side matches the extra contribution to the factor \eqref{eq:frontfactor}
 in front of the sum in the series for $\Gamma_\circ$.
We conclude that the contribution of $\ell$ in the series for $\Gamma$ equals the sum over $w\in W$ of the contributions of $R_w(\ell)$ in the series for $\Gamma_\circ$.
Since the maps $R_w$ for $w\in W$ are injective, the statement for this move follows.

\bigskip

\noindent \textit{Step (B$+$): The Neumann move (B$+$) from Figure \ref{fig:Neumann}.}
In this case, one has
\[
\sigma_\circ=1+\sigma, \qquad \pi_\circ = 1+ \pi, \qquad 3\sigma_\circ - \mathrm{tr}\, B_\circ = 1+ 3\sigma - \mathrm{tr}\, B. 
\]
We conclude that the factor \eqref{eq:frontfactor} in front of the sum in the series for $\Gamma_\circ$ has an extra factor $(-1)^{|\Delta^+|}q^{\frac{1}{2}\langle \rho, \rho\rangle}$.

For a choice of $w\in W$, consider the function 
\begin{equation}
\label{eq:RwB+}
R_w\colon L'\otimes_{\mathbb{Z}} Q \rightarrow L'_\circ \otimes_{\mathbb{Z}} Q, \qquad
(a_\sharp, a_1)\mapsto (a_\sharp, a_1+2w(\rho),  2 w(\rho)).
\end{equation}
For $a\in \delta +2L' \otimes_{\mathbb{Z}} Q$, define 
$a_\circ:= R(a)\in \delta_\circ +2L'_\circ \otimes_{\mathbb{Z}} Q$ where $R=R_w$ with $w=1_W$.

As with the previous move, we can restrict the sum in the series for $\Gamma_\circ$ over only those $\ell_\circ$ which are of type $\ell_\circ= R_w(\ell)$ for some $\ell\in a + 2B L\otimes_{\mathbb{Z}} Q$ and some $w\in W$. Also, let $n$ be the degree of the vertex in $\Gamma_\circ$ adjacent to the added vertex. 
As with the previous move, the assumption that this Neumann move does not create a reducible vertex implies $n\neq 3$, and thus there are no new forcing bridges.
Hence the sets of Weyl assignments for the two plumbing trees are isomorphic.

For $\ell\in a + 2B L\otimes_{\mathbb{Z}} Q$, write $\ell=(\ell_\sharp, \ell_1)$. Select a Weyl assignment $\xi$ on $\Gamma$, let $\xi_\circ$ be the corresponding Weyl assignment on $\Gamma_\circ$, and let $x\in W$ be the value of $\xi$ at the vertex in $\Gamma_\circ$ adjacent to the added vertex.
From  \ref{it:P2}, one has
\begin{align}
\label{eq:moveB+Lemma2app}
\begin{split}
c_{x,n-1}(\ell_1) &= 
(-1)^{|\Delta^+|}\sum_{w\in W} (-1)^{\ell(w)}\, c_{x,n}\left( \ell_1 + 2 w(\rho)\right) \\
&= (-1)^{|\Delta^+|} \sum_{w\in W} c_{1_W,1}\left( 2w(\rho)\right) c_{x,n}\left( \ell_1 + 2 w(\rho)\right).
\end{split}
\end{align}
The second identity follows from \eqref{eq:Px1}. 
Multiplying both sides of \eqref{eq:moveB+Lemma2app} by the contributions corresponding to the remaining vertices of $\Gamma$, one has
\begin{equation*}
\prod_{v\in V(\Gamma)} \left[ P_{\xi_v, \deg v}(z_v) \right]_{\ell_v}
=  (-1)^{|\Delta^+|}
\sum_{w\in W} 
\prod_{v\in V(\Gamma_\circ)} \left[ P_{\xi_\circ(v), \deg v}(z_v) \right]_{R_w(\ell)_v}.
\end{equation*}
This is
\begin{equation}
\label{eq:f=sumfB+}
c_{\Gamma, \xi}(\ell) =  (-1)^{|\Delta^+|} \sum_{w\in W}c_{\Gamma_\circ, \xi_\circ}(R_w(\ell)).
\end{equation}

For $w\in W$, a direct computation shows that 
\[
B^{-1}
\ell
 =
\left(h_\sharp,  h_1\right)
\quad \Rightarrow \quad 
B_\circ^{-1}
R_w(\ell)
=
\left(h_\sharp, h_1,2w(\rho) -h_1 \right).
\]
This implies that 
\begin{align*}
\langle R_w(\ell), R_w(\ell) \rangle &= 
\left( \ell_\sharp , \ell_1+2w(\rho), 2w(\rho) \right)^t
\left(h_\sharp, h_1, 2w(\rho) -h_1\right)\\
&=\langle \ell, \ell \rangle +4 \langle \rho, \rho \rangle.
\end{align*}
Thus from \eqref{eq:f=sumfB+}, we have
\begin{align}
\begin{split}
\label{eq:conteqB+}
&(-1)^{|\Delta^+|\, \pi}
c_{\Gamma, \xi}(\ell) \, q^{-\frac{1}{8}\langle \ell, \ell\rangle}\\
&= 
(-1)^{|\Delta^+|\, \pi_\circ}
q^{\frac{1}{2}\langle \rho,\rho\rangle}
\sum_{w\in W}c_{\Gamma_\circ, \xi_\circ}(R_w(\ell))\,
q^{-\frac{1}{8}\langle R_w(\ell), R_w(\ell)\rangle}.
\end{split}
\end{align}
The factor $q^{\frac{1}{2}\langle \rho,\rho\rangle}$ on the right-hand side matches the extra contribution to $q$ in the factor \eqref{eq:frontfactor}
 in front of the sum in the series for $\Gamma_\circ$.
We conclude that the contribution of $\ell$ in the series for $\Gamma$ equals the sum over $w\in W$ of the contributions of $R_w(\ell)$ in the series for $\Gamma_\circ$.
Since the maps $R_w$ for $w\in W$ are injective, the statement for this move follows.

\bigskip

\noindent \textit{Step (C): The Neumann move (C) from Figure \ref{fig:Neumann}.}
In this case, one has
\[
\sigma_\circ=\sigma, \qquad \pi_\circ = 1+ \pi, \qquad 3\sigma_\circ - \mathrm{tr}\, B_\circ = 3\sigma - \mathrm{tr}\, B. 
\]
We conclude that the factor \eqref{eq:frontfactor} in front of the sum in the series for $\Gamma_\circ$ has an extra factor $(-1)^{|\Delta^+|}$.
 
Recall the function $R_\beta$ with $\beta\in Q$ from \eqref{eq:RaC}:
\begin{equation*}
R_\beta\colon L'\otimes_{\mathbb{Z}} Q \rightarrow L'_\circ \otimes_{\mathbb{Z}} Q, \qquad
(a_\sharp, a_0, a_\flat)\mapsto (a_\sharp,  a_0+\beta, 0, \beta, -a_\flat)
\end{equation*}
 where the entry $a_0$ corresponds to the vertex in $\Gamma$ labelled by $m_1+m_2$, the entries $a_0+\beta$, $0$, and $\beta$ correspond to the vertices in $\Gamma_\circ$ labelled by $m_1$, $0$, and $m_2$, respectively, and the subtuples $a_\sharp$ and $a_\flat$ correspond to all the vertices in $\Gamma_\circ$ on their left and right, respectively.

For $a\in \delta +2L' \otimes_{\mathbb{Z}} Q$, define $a_\circ\in \delta_\circ +2L'_\circ \otimes_{\mathbb{Z}} Q$ as
\[
a_\circ:= \left\{
 \begin{array}{ll}
R_{2\rho}(a) & \mbox{if $\deg(v_1)\equiv \deg(v_2)$ mod $2$,}\\[5pt]
R_{0}(a) & \mbox{otherwise.}
 \end{array}
\right.
\]
This is as in \eqref{eq:gammadef}.

As the  vertex labelled by $0$ in $\Gamma_\circ$ has degree $2$, from \ref{it:P1} we deduce that for $\ell_\circ\in a_\circ + 2B_\circ L_\circ\otimes_{\mathbb{Z}} Q$,
one has $c_{\Gamma_\circ, \xi_\circ}(\ell_\circ)=0$ when $\ell_\circ$ has a non-zero  component corresponding to the vertex of $\Gamma_\circ$ labelled by $0$. Hence, we can restrict the sum in the series for $\Gamma_\circ$ over only those $\ell_\circ$ which are of type $\ell_\circ= R_\beta(\ell)$ for some $\ell\in a + 2B L\otimes_{\mathbb{Z}} Q$ and some $\beta\in Q$. Note that for $\ell\in a + 2B L\otimes_{\mathbb{Z}} Q$ and $\beta\in Q$, one has
\[
R_\beta(\ell) \in a_\circ + 2B_\circ L_\circ\otimes_{\mathbb{Z}} Q
\]
if and only if $\beta\in \beta_0 +2Q$ with $\beta_0$ defined as in \eqref{eq:gammadef}.

Select a  Weyl assignment $\xi$ for $\Gamma$, and define a  Weyl assignment $\xi_\circ$ for $\Gamma_\circ$ such that, 
for a vertex $v$ with $\deg v\geq 3$, one has
\begin{equation}
\label{eq:identifXiC}
\xi_\circ \colon v \mapsto \left\{
\begin{array}{ll}
\xi_v & \mbox{if $v$ is on the left of $v_1$,}\\[0.5pt]
\iota\xi_v & \mbox{if $v$ is on the right of $v_2$.} 
\end{array}
\right.
\end{equation}
Here $\iota$ is as in \eqref{eq:iota}, and  $v_1$ and $v_2$ are the vertices labelled by $m_1$ and $m_2$ in $\Gamma_\circ$.
Moreover, define
\[
\xi_{\circ, v_1} := \left\{
\begin{array}{ll}
\xi_{v_0} & \mbox{if $\deg(v_1)\geq 3$,}\\[5pt]
 1_W & \mbox{if $\deg(v_1)\leq 2$,}
\end{array}
\right.
\]
and
\[
\xi_{\circ, v_2} :=\left\{
\begin{array}{ll}
 \iota \xi_{v_0} & \mbox{if $\deg(v_2)\geq 3$,}\\[5pt]
 1_W & \mbox{if $\deg(v_2)\leq 2$,}
\end{array}
 \right.
\]
where $v_0$ is  the vertex labelled by $m_1+m_2$ in $\Gamma$.
The value of $\xi_\circ$ at the  vertex labelled by $0$ in $\Gamma_\circ$ is $1_W$, as determined by \eqref{eq:Wassconditiondeg012}.
The map $\xi\mapsto \xi_\circ$ is the natural isomorphism of the sets of Weyl assignments for $\Gamma$ and $\Gamma_\circ$.

Let $p$ and $q$ be the degrees of the vertices $v_1$ and $v_2$ in $\Gamma_\circ$, respectively. Then the degree of the vertex $v_0$ in $\Gamma$ is $p+q-2$. 
Recall that we are assuming that the Neumann move does not create a new reducible vertex.
Moreover, when $p,q\geq 3$, the tree $\Gamma_\circ$ has one more forcing bridge with respect to $\Gamma$, and the definition of $\xi_\circ$ is compatible with the condition \eqref{eq:forcingbridgescondition}.

For $\ell\in a + 2B L\otimes_{\mathbb{Z}} Q$, write $\ell=(\ell_\sharp, \ell_0, \ell_\flat)$.
Let $x:= \xi(v_0)\in W$.
From \ref{it:P4}, one has
\[
c_{x,p+q-2}(\ell_0) = \sum_{\beta\in \beta_0+2Q} c_{x,p}\left( \ell_0 + \beta\right) c_{x,q}\left( -\beta\right).
\]
Applying \ref{it:P3}, one has
\begin{equation}
\label{eq:Ctempid}
c_{x,p+q-2}(\ell_0) = (-1)^{|\Delta^+|\, q}\sum_{\beta\in \beta_0+2Q} c_{x,p}\left( \ell_0 + \beta\right) c_{\iota x,q}\left(\beta\right).
\end{equation}
Let $V_\flat(\Gamma_\circ)$ be the set of all vertices of $\Gamma_\circ$ on the right of the vertex labelled by $m_2$.
Applying \ref{it:P3} to all contributions corresponding to vertices in $V_\flat(\Gamma_\circ)$ and using that $q+\sum_{v\in V_\flat(\Gamma_\circ)} \deg v$ is odd (this is as in \eqref{eq:oddsumdegv}), one has
\begin{multline*}
(-1)^{|\Delta^+|\, q} c_{\iota x,q}\left(\beta\right) \prod_{v\in \mathrm{V_\flat(\Gamma_\circ)}} c_{\xi_v,\deg v}(\ell_v) \\
=
(-1)^{|\Delta^+|} c_{\iota x,q}\left(\beta\right) \prod_{v\in \mathrm{V_\flat(\Gamma_\circ)}} c_{\iota \xi_v,\deg v}(-\ell_v).
\end{multline*}
Multiplying both sides of \eqref{eq:Ctempid} by the contributions corresponding to the remaining vertices, using \ref{it:P1} and the last identity, one has
\begin{equation}
\label{eq:xif=sumxifC}
\prod_{v\in V(\Gamma)} \left[ P_{\xi_v, \deg v}(z_v) \right]_{\ell_v}
= (-1)^{|\Delta^+|}
\sum_{\beta\in \beta_0+2Q} 
\prod_{v\in V(\Gamma_\circ)} \left[ P_{\xi_\circ(v), \deg v}(z_v) \right]_{R_\beta(\ell)_v}.
\end{equation}
This is
\begin{equation}
\label{eq:f=sumfC}
c_{\Gamma, \xi}(\ell) =  (-1)^{|\Delta^+|} \sum_{\beta\in \beta_0+2Q}c_{\Gamma_\circ, \xi_\circ}(R_\beta(\ell)).
\end{equation}
The factor $(-1)^{|\Delta^+|}$ on the right-hand side matches the extra contribution to the factor \eqref{eq:frontfactor}
 in front of the sum in the series for $\Gamma_\circ$.
 
For $\beta\in Q$, a direct computation shows that 
\[
B^{-1}
\ell
 =
\left(h_\sharp,  h_0, h_\flat\right)
\quad \Rightarrow \quad 
B_\circ^{-1}
R_\beta(\ell)
=
\left(h_\sharp, h_0, h'_0, -h_0, -h_\flat \right)
\]
for some $h'_0\in Q$.
This implies that 
\begin{equation}
\label{eq:RRllC}
\langle R_\beta(\ell), R_\beta(\ell) \rangle = 
(\ell_\sharp,  \ell_0+\beta, 0, \beta, -\ell_\flat)^t
\left(h_\sharp, h_0, h'_0, -h_0, -h_\flat \right)
=\langle \ell, \ell \rangle .
\end{equation}

We conclude that 
\begin{equation}
\label{eq:conteqC}
(-1)^{|\Delta^+|\, \pi}
c_{\Gamma,\xi}(\ell) 
q^{-\frac{1}{8}\langle \ell, \ell\rangle}=  
(-1)^{|\Delta^+|\, \pi_\circ} \sum_{\beta\in \beta_0+2Q}c_{\Gamma_\circ, \xi_\circ}(R_\beta(\ell))
q^{-\frac{1}{8}\langle R_\beta(\ell), R_\beta(\ell)\rangle}.
\end{equation}
Hence the contribution of $\ell$ in the series for $\Gamma$ equals the sum over $\beta\in \beta_0+2Q$ of the contributions of $R_\beta(\ell)$ in the series for $\Gamma_\circ$.
Since the maps $R_\beta$ are injective for all $\beta$, the statement for this move follows. 

This concludes the proof.
\end{proof}


\section{Strong characterization of admissibility and symmetry}
\label{sec:strong}

Here we prove a strong characterization of the admissible and symmetric collections and conclude with the proof of Theorem  \ref{thm:finalthmintro}. Note that when $\Gamma$ fails to be reduced and refinable, the series $\mathsf{Y}\left(q\right)=\mathsf{Y}_{\mathcal{P},S,a}\left(q\right)$ may fail to exist, i.e., $\mathsf{Y}\left(q\right)=\mathsf{Y}_{\mathcal{P},S,a}\left(q\right)$ may not be well-defined (see \cite[\S 4.2]{MT2}).

\begin{theorem}
\label{thm:qseriesinvariance}
When it exists,
the series $\mathsf{Y}\left(q\right)=\mathsf{Y}_{\mathcal{P},S,a}\left(q\right)$ obtained from a collection $\mathcal{P}$ satisfying \eqref{eq:Px1} and a subset $S\subseteq W^{V(\Gamma)}$ is:
\begin{enumerate}[(i)]
\item  invariant under the five Neumann moves amongst reduced plumbing trees if and only if $\mathcal{P}$ is  admissible  and $S\subseteq \Xi$, and 

\item additionally invariant under the action of the Weyl group $W$ on $a$, i.e.,
\[
\mathsf{Y}_{\mathcal{P},S,a}\left(q\right)=\mathsf{Y}_{\mathcal{P},w(S),w(a)}\left(q\right), \quad \mbox{for $w\in W$,}
\]
if and only if $\mathcal{P}$ is  admissible and symmetric and $S=\Xi$.
\end{enumerate}
\end{theorem}

\begin{proof}
After Theorems \ref{thm:seriesinvariancef} and \ref{thm:Winvqseries},
it remains to prove the ``only if'' statements.
For the ``only if'' in part (i), assume $\mathsf{Y}\left(q\right)$ is invariant under the five Neumann moves amongst reduced plumbing trees. 
We prove that the collection $\mathcal{P}$ is necessarily {admissible}, i.e., $\mathcal{P}$ satisfies properties \ref{it:P1}--\ref{it:P4}.
We use notation as in the proof of Theorem \ref{thm:seriesinvariancef}. 

For each Neumann move  in Figure \ref{fig:Neumann}, the vertex set of $\Gamma$ is naturally identified with a subset of the vertex set of $\Gamma_\circ$.
The bijection of the generalized $\mathrm{Spin}^c$-structures from \eqref{eq:isoBQM} is induced 
by the map $R$ defined in \eqref{eq:Ra}--\eqref{eq:gammadef}, see Proposition \ref{prop:BQMinvariance2}.
Depending on the Neumann move, $R$ might entail a minus sign on the entries corresponding to the right portion of $\Gamma_\circ$. 
The map $R$ induces a map $\xi\mapsto \xi_\circ$ 
such that $\xi_\circ(v)=\xi(v)$ if $R(\ell)_v=\ell_v$, while $\xi_\circ(v)=\iota\xi(v)$ if $R(\ell)_v=-\ell_v$.
Thus the map $\xi\mapsto \xi_\circ$ is as in \eqref{eq:identifXiA+} and \eqref{eq:identifXiC}.
The value of $\xi_\circ$ at a new vertex of $\Gamma_\circ$ of degree $1$ or $2$ could possibly be arbitrary.

Assume $\mathsf{Y}\left(q\right)$ is invariant under the Neumann move (A$-$). 
An element $\ell_\circ$ in the image of the map $R$ in \eqref{eq:Ra} has component $0$ corresponding to the added vertex $v_0$ in $\Gamma_\circ$.
From the definition of the coefficient $c_{\Gamma, \xi}(\ell)$ in \eqref{eq:cGammaxiell}, one deduces 
\[
\left[ P_{\xi_\circ(v), 2}(z_v)\right]_0 = 1
\]
for $v=v_0$. After possibly removing contributions that sum to zero, it follows that \ref{it:P1} holds.

Next, assume $\mathsf{Y}\left(q\right)$ is also invariant under the Neumann move (A$+$). 
Then necessarily, one has
\begin{equation*}
c_{\Gamma, \xi}(\ell) = (-1)^{|\Delta^+|} c_{\Gamma_\circ, \xi_\circ}(R(\ell))
\end{equation*}
as in \eqref{eq:fgammaB+}. Let $v_2$ be the vertex labelled by $m_2$ in $\Gamma$. 
By induction on the degree of  $v_2$ across all plumbing trees $\Gamma$, one deduces that 
\begin{equation*}
c_{x,n}(\alpha) = (-1)^{|\Delta^+|\,n} c_{\iota x,n}(-\alpha)  \qquad \mbox{for $\alpha\in Q$}
\end{equation*}
where $x=\xi_v$ and $n=\deg v$ for $v$ equal to $v_2$ or a vertex on its right. Hence \ref{it:P3} holds.

Assume $\mathsf{Y}\left(q\right)$ is invariant under the Neumann move (B$-$). 
From the assumption \eqref{eq:Px1}, the sum in the series for $\Gamma_\circ$ is over elements 
$\ell_\circ$ whose component corresponding to the added vertex $v_0$ is in the orbit $-2W(\rho)$, 
i.e., $\ell_\circ$ is in the image of  $R_w$ for some $w\in W$ as in \eqref{eq:RwB-}, up to removing contributions that sum to zero. 
As in \eqref{eq:moveB-Lemma2app} and using the same notation as in there, one then has
\begin{align*}
c_{x,n-1}(\ell_1) = \sum_{w\in W} c_{\xi_\circ(v_0),1}\left( -2w(\rho)\right) c_{x,n}\left( \ell_1 + 2 w(\rho)\right).
\end{align*}
Together with the assumption \eqref{eq:Px1}, this implies \ref{it:P2}.

Finally, assume $\mathsf{Y}\left(q\right)$ is also invariant under the Neumann move (C). 
Then necessarily, one has
\begin{equation*}
c_{\Gamma, \xi}(\ell) =  (-1)^{|\Delta^+|} \sum_{\beta\in \beta_0+2Q}c_{\Gamma_\circ, \xi_\circ}(R_\beta(\ell))
\end{equation*}
as in \eqref{eq:f=sumfC}. By induction on the degree of the vertex labelled by $m_2$ in $\Gamma_\circ$ and combining with \ref{it:P3}, now verified, one deduces that 
\[
c_{x,p+q-2}(\ell_0) = \sum_{\beta\in \beta_0+2Q} c_{x,p}\left( \ell_0 + \beta\right) c_{x,q}\left( -\beta\right)
\]
hence \ref{it:P4} holds.

\smallskip

Next, we show that $S\subseteq \Xi$.
Assume instead $W^s\supseteq S\supsetneq \Xi$. 
From Theorem \ref{thm:adm}, one has
\begin{equation*}
\{P_{x,n}(z) \,|\, x\in W,\, n\geq 0\} \subseteq \{K_{w,n}(z) \,|\, w\in W,\, n\geq 0\}.
\end{equation*}
Since $P_{x,n}(z)$ is independent of $x$ for $n\in\{0,1,2\}$,
after possibly replacing $S$ without affecting the series, we can assume that \eqref{eq:Wassconditiondeg012} holds.
Then the assumption $S\supsetneq \Xi$ implies that for a plumbing tree $\Gamma$ there exists a forcing bridge between vertices $u$ and $v$ and $\xi\in S$ such that the values $\xi_u$ and $\xi_v$ are not coordinated as in \eqref{eq:forcingbridgescondition}. 
Then contracting this forcing bridge via a sequence  of the Neumann moves results in the evaluation of a product $K_{x,n}(z)K_{y,m}(z)$ with $x\neq y$ and $n,m\geq 3$. Since the product $K_{x,n}(z)K_{y,m}(z)$ with $n,m\geq 3$ is defined only for $x=y$, this leads to a contradiction. It follows that necessarily $S\subseteq\Xi$, hence part~(i).

\smallskip

Next, we prove that if $\mathsf{Y}\left(q\right)$ is also invariant under the action of $W$, then $\mathcal{P}$ is necessarily symmetric, i.e., $\mathcal{P}$ satisfies \ref{it:P5}.
The case $n=2$ is trivially satisfied by \ref{it:P2}, now verified, and the case $n=1$ follows from the assumption \eqref{eq:Px1}. Similarly, the case $n=0$ follows from the fact that $P_{x,0}(z)$ is uniquely determined as in Remark \ref{rmk:n012}.
Also, as property \ref{it:P4} implies $P_{x,n}(z)= \left(P_{x,3}(z) \right)^{n-2}$ for $n>3$, it remains to discuss the case $n=3$. 
This follows from the invariance under $W$ of the series $\mathsf{Y}_{a}\left(q\right)$ for plumbing trees with exactly one vertex of degree $3$.
Hence \ref{it:P5} holds.

\smallskip

It remains to verify that necessarily $S=\Xi$. 
From Theorem \ref{thm:uniquenessadm}, it follows that $\mathcal{P}=\mathcal{K}$.
Then assuming $S\subsetneq \Xi$ contradicts the invariance of the series $\mathsf{Y}\left(q\right)$ under $W$ (unless the series for $S$ and $\Xi$ coincide).
It follows that the series is obtained from the case $S=\Xi$,  hence part (ii).
\end{proof}

\begin{proof}[Proof of Theorem \ref{thm:finalthmintro}]
The statements follow from Theorems \ref{thm:qseriesinvariance} and~\ref{thm:uniquenessadm}.
\end{proof}

\begin{remark}
As in Remark \ref{rmk:admissibleAJK}, it is interesting to compare this statement with the results from \cite{akhmechet2023lattice}: 
To obtain invariance only under the Neumann moves between negative-definite plumbing trees, i.e., the Neumann move (A$-$) and (B$-$), one only requires that the collections $\mathcal{P}$ with \eqref{eq:Px1} satisfy  \ref{it:P1} and \ref{it:P2}. In this context, there are infinitely many such collections which additionally satisfy \ref{it:P3}, yielding infinitely many invariant series for the negative-definite case. However, to obtain invariance under the Neumann move (C) imposes the additional condition \ref{it:P4} on the collections $\mathcal{P}$, yielding only finitely many invariant series for plumbing trees in general.
\end{remark}


\section{Examples}
\label{sec:examples}

We discuss here the series $\mathsf{Y}\left(q\right)$ in the cases of plumbing trees with either one or two vertices of degree at least $3$.

From Theorem \ref{thm:qseriesinvariance}(i) and Remark \ref{rmk:P=K}, we may assume $\mathcal{P}=\mathcal{K}$, the Kostant collection from \S\ref{sec:Kz}. We will thus omit $\mathcal{P}$ from the notation.

\subsection{Brieskorn spheres}
We compute here the series $\mathsf{Y}_{S,a}\left(q\right)$ in the case of Brieskorn homology spheres for 
$S \subset \Xi$ of size one. 
We show that all such series are identical and equal to the series $\widehat{Z}_a(q)$ computed for $Q=A_1$ in \cite{gukov2021two} and for arbitrary $Q$ in \cite{park2020higher}. Thus the fact that their average recovers 
 $\widehat{Z}_a(q)$ is trivially satisfied in this case. 

For this, we use the fact that a Brieskorn homology sphere is realized as a negative-definite manifold constructed from a  star-shaped plumbing tree $\Gamma$ with central vertex of degree $3$. 
Since $\Gamma$ is negative definite, one has $\pi=0$ and $\sigma=-s$.

Select a root lattice $Q$. As $\Gamma$ has only one vertex of degree $\geq 3$, the set of Weyl assignments from \S\ref{sec:chamberass} is $\Xi\cong W$. Choose an order of the basis of $L'\otimes_{\mathbb{Z}} Q\cong Q^s$ so that for $f\in Q^s$ one writes
\begin{equation*}
f=(f_0, f_1, f_2, f_3, \dots)
\end{equation*}
with $f_0$ corresponding to the vertex of $\Gamma$ of degree $3$ and $f_1, f_2, f_3$ to the three vertices of degree $1$.

Since $H_1(M; Q)=0$, one has that $a=\delta$ from \eqref{eq:delta} is the unique generalized $\mathrm{Spin}^c$-structure. 
This is $\delta=(-1,1,1,1,0\dots,0)\otimes 2\rho$. We will omit it from the notation.

For a subset $S=\{x\}\subset W \cong\Xi$, the series from \eqref{eq:Ypa} is
\begin{equation}
\label{eq:YPqBrieskorn}
\mathsf{Y}_{x}(q) = 
q^{-\frac{1}{2}(3s+\mathrm{tr}\,B)\langle \rho, \rho \rangle} 
\sum_f c_{\Gamma,x} (f)\,q^{-\frac{1}{8}\langle f, f\rangle}
\end{equation}
where the sum is over  $f\in \delta + 2BQ^s\subset Q^s$. 
Recall the definition of the coefficients  $c_{\Gamma,x} (f)$ in \eqref{eq:cGammaxiell}.
Since $\Gamma$ is a star-shaped tree with three legs, 
the sum over $f$ in \eqref{eq:YPqBrieskorn} can be restricted to those $f$ which are of the form 
\begin{align*}
f=(\gamma, 2w_1(\rho), 2w_2(\rho), 2w_3(\rho), 0,\dots,0)\in Q^s \\
\mbox{with $\gamma\in 2\rho+2Q$ and $w_1,w_2,w_3\in W$}
\end{align*}
as $c_{\Gamma,x} (f)$ vanishes otherwise.
Next, we compute the contribution of each entry of $f$ to $c_{\Gamma,x} (f)$.
Writing $K(z)=\sum_{\alpha\in Q}d(\alpha)\,z^\alpha$, from \eqref{eq:Ktwist} one has
\[
K_x(z)=(-1)^{\ell(x)}\sum_{\alpha\in Q}d\left(x^{-1}\alpha\right)\,z^\alpha.
\]
Hence for a fixed $x\in S$, the contribution of $f_0=\gamma$ is $(-1)^{\ell(x)}\,d\left(x^{-1}\gamma\right)$.
Applying \eqref{eq:Px1}, the contribution of $f_i=2w_i(\rho)$ for $i=1,2,3$ is $(-1)^{\ell(w_i)}$. Multiplying these contributions, one has
\[
c_{\Gamma,x} (f)=(-1)^{\ell(xw_1w_2w_3)} \,d\left(x^{-1}\gamma\right).
\]
Thus the series  becomes
\[
\mathsf{Y}_{x}(q) = 
q^{-\frac{1}{2}(3s+\mathrm{tr}\,B)\langle \rho, \rho \rangle} 
\sum_f  (-1)^{\ell(xw_1w_2w_3)} d\left(x^{-1}\gamma\right)q^{-\frac{1}{8}\langle f, f\rangle}.
\]
By making use of the symmetry 
\[
f\mapsto x^{-1}f=\left(x^{-1}\gamma, 2x^{-1}w_1(\rho), 2x^{-1}w_2(\rho), 2x^{-1}w_3(\rho), 0,\dots,0\right)
\]
and the fact that $\langle f, f\rangle=\langle x^{-1}f, x^{-1}f\rangle$ for $x\in W$, the series can be rewritten as
\[
\mathsf{Y}_{x}(q) = 
q^{-\frac{1}{2}(3s+\mathrm{tr}\,B)\langle \rho, \rho \rangle} 
\mathop{\sum_{\gamma\in Q}}_{w_1, w_2, w_3\in W}  (-1)^{\ell(w_1w_2w_3)} d\left(\gamma\right)q^{-\frac{1}{8}\langle f, f\rangle}.
\]
This uses that $\ell(x)=\ell(x^{-1})$ and thus ${\ell(xw_1w_2w_3)}\equiv {\ell(x^{-1}w_1 x^{-1}w_2 x^{-1}w_3)}$ mod~$2$.

As a first observation, the right-hand side of this formula is independent of $x$. It follows that the series $\mathsf{Y}_{x}(q)$ for $x\in W$ are all equal. Moreover, this formula recovers the series $\widehat{Z}(q)$ computed for Brieskorn homology spheres in \cite{gukov2021two} and \cite{park2020higher}. 

The fact that all the series $\mathsf{Y}_{x}(q)$ are identical in this example is due to the symmetry of the Weyl group and the fact that $\Xi\cong W$ here. 
For more details on this symmetry, see the similar case treated in \S\ref{sec:Seifert}.
For an example where the series  $\mathsf{Y}_{x}(q)$ vary with $x\in \Xi$, see \S\ref{sec:nonSeifert}.

\begin{figure}[t]
\[
\begin{tikzpicture}[baseline={([yshift=0ex]current bounding box.center)}]
      \path(0,0) ellipse (2 and 2);
      \tikzstyle{level 1}=[counterclockwise from=90, level distance=40mm, sibling angle=120]
      \node [draw, circle, fill, inner sep=1.5, label={[label distance=10]30:$-1$}] (A0) at (0:0) {}
            child {node [draw, circle, fill, inner sep=1.5, label={[label distance=10]180:$-2$}]{}}
            child {node [draw, circle, fill, inner sep=1.5, label={[label distance=10]180:$-3$}]{}}
	    child {node [draw, circle, fill, inner sep=1.5, label={[label distance=10]0:$-7$}]{}};
    \end{tikzpicture}
\]
\caption{The Brieskorn sphere $\Sigma(2,3,7)$.}
\label{fig:S237}
\end{figure}

\subsection{The Brieskorn sphere $\Sigma(2,3,7)$} Here we consider the case $M=\Sigma(2,3,7)$ in detail. 
This is obtained from the negative-definite plumbing tree in Figure \ref{fig:S237}.
For $Q=A_1$, one has
\[
\mathsf{Y}_{x}(q) = q^{1/2}\left(1 -q -q^5 +q^{10} -q^{11} +q^{18} +q^{30} -q^{41} + q^{43} -q^{56} + O(q^{76})\right) 
\]
and for $Q=A_2$, one has
\[
\mathsf{Y}_{x}(q) = q^2\left(1 -2q +2q^3 +q^4 -2q^5 -2q^8 +4q^9 +2q^{10} -4q^{11} +O(q^{13})\right).
\]
Both formulae are independent of $x\in \Xi$ and recover the series $\widehat{Z}(q)$ computed  in \cite{gukov2021two} and \cite{park2020higher}, respectively.

\subsection{Seifert manifolds}
\label{sec:Seifert}
Here we consider the case when $M$ is more generally a Seifert manifold. In this case, the plumbing tree $\Gamma$ consists of a star-shaped tree. Furthermore, we assume that (the class of) $a$ is invariant under $W$ 
(e.g., when $a$ is the unique self-conjugate $\mathrm{Spin}^c$-structure on $M$). 

As for Brieskorn spheres, $\Xi\cong W$. From \eqref{eq:Wxwa} and the assumption that $w(a)=a$ for all $w\in W$,
it follows that the series $\mathsf{Y}_{x,a}(q)$ for all $x\in W$ are identical and thus equal to $\widehat{Z}_a(q)$, as for Brieskorn spheres.

\begin{figure}[t]
\[
\begin{tikzpicture}[baseline={([yshift=0ex]current bounding box.center)}]
      \path(0,0) ellipse (2 and 2);
      \tikzstyle{level 1}=[counterclockwise from=135, level distance=40mm, sibling angle=90]
      \node [draw, circle, fill, inner sep=1.5, label={[label distance=10]60:$-2$}] (A1) at (180:2) {}
            child {node [draw, circle, fill, inner sep=1.5, label={[label distance=10]180:$-2$}]{}}
	    child {node [draw, circle, fill, inner sep=1.5, label={[label distance=10]180:$-3$}]{}};
      \tikzstyle{level 1}=[counterclockwise from=45,level distance=40mm,sibling angle=-90]
      \node [draw, circle,fill, inner sep=1.5, label={[label distance=10]-120:$-2$}] (A2) at (0:2) {}
	   child {node [draw, circle, fill, inner sep=1.5, label={[label distance=10]0:$-2$}]{}}
	    child {node [draw, circle, fill, inner sep=1.5, label={[label distance=10]0:$-3$}]{}};
      \tikzstyle{level 1}=[counterclockwise from=150,level distance=15mm,sibling angle=30]
      \path (A1) edge []  node[midway, label={[label distance=10]-90:}]{} (A2);
    \end{tikzpicture}
\]
\caption{A non-Seifert manifold.}
\label{fig:nonSeifert}
\end{figure}

\subsection{A non-Seifert manifold}
\label{sec:nonSeifert}
Here we consider the 3-manifold $M$ obtained from the negative-definite plumbing tree in Figure \ref{fig:nonSeifert}.

One has $H_1(M; \mathbb{Z})=\mathbb{Z}_{13}$. 
For simplicity, we do not separate between the representatives of the various $\mathrm{Spin}^c$-structures, i.e., we consider the series 
\begin{align*}
\widehat{Z}_{\mathrm{tot}}(q) &:= \sum_{a\in\mathrm{Spin}^c(M) }\widehat{Z}_{a}(q), &
\mathsf{Y}_{S, \mathrm{tot}}\left(q\right) &:= \sum_{a\in\mathrm{Spin}^c(M) } \mathsf{Y}_{S,a}\left(q\right).
\end{align*}

For $Q=A_1$, the first few terms of the series $\widehat{Z}_{\mathrm{tot}}(q)$ for $M$ have been computed in \cite[(3.161)]{gukov2020bps}:
\begin{eqnarray*}
\widehat{Z}_{\mathrm{tot}}(q) &= & \frac{1}{4}\Bigg(
q^{-1/2}\left(2-2q+2q^2  \right)
+2q^{5/26}\left(-3+2q +2q^2 -4q^3\right)\\
&& +2q^{7/26}\left(4 +q  \right)
+2q^{-7/26}\left(-3 -3q^2  \right)\\
&&+2q^{-11/26}\left(-1 +2q -2q^2+4q^3 \right)
+2q^{-5/26}\left(2 +2q^2 -q^3 \right)\\
&&+2q^{-15/26}\left( -1 -2q+ 2q^2  \right) + O\left(q^{85/26}\right)
\Bigg).
\end{eqnarray*}
Here $O(q^x)$ stands for $q^x$ times a series in non-negative powers of $q$. (The series is approximated to a higher degree in \cite{gukov2020bps}. Due to an evident typo,  the factors $+2$ multiplying the rational powers of $q$, including the sign, are missing in \cite{gukov2020bps}.)

The series $\widehat{Z}_{\mathrm{tot}}(q)$ decomposes as the average of the series $\mathsf{Y}_{S, \mathrm{tot}}\left(q\right)$ for 
$S \subset \Xi$ of size one. 
As $\Gamma$ has exactly two vertices of degree $\geq 3$ and no forcing bridges, the set of Weyl assignments from \S\ref{sec:chamberass} is $\Xi\cong W^2$. Choose an order of the vertices of $\Gamma$ and write $\xi\in \Xi$ as $\xi=(x,y)$ with $x$ and $y$ being the assignment of $\xi$ to the left and right vertices of degree $3$, respectively. 

For $Q=A_1$, one has $\xi\in \{\pm 1\}^2$.
From \eqref{eq:Wxwa}, one has 
\begin{align*}
\mathsf{Y}_{(1,1), \mathrm{tot}}\left(q\right)&=\mathsf{Y}_{(-1,-1), \mathrm{tot}}\left(q\right), & \mathsf{Y}_{(1,-1), \mathrm{tot}}\left(q\right)&=\mathsf{Y}_{(-1,1), \mathrm{tot}}\left(q\right).
\end{align*}
A direct computation yields
\begin{eqnarray*}
\mathsf{Y}_{(1,1), \mathrm{tot}}(q) &= & 
q^{-1/2}\left(1+q^2  \right)
+q^{5/26}\left(-2 \right)
+q^{-11/26}\left( 2q^3 \right)
+q^{-5/26}\left(-2q \right)\\
&&+q^{-15/26}\left(  2q^3 \right) 
+ O\left(q^{85/26}\right)
\end{eqnarray*}
and
\begin{eqnarray*}
\mathsf{Y}_{(1,-1), \mathrm{tot}}(q) &= & 
q^{-1/2}\left(-q \right)
+q^{5/26}\left(-1+2q +2q^2 -4q^3\right)\\
&& +q^{7/26}\left(4 +q  \right)
+q^{-7/26}\left(-3 -3q^2  \right)\\
&&+q^{-11/26}\left(-1 +2q -2q^2+2q^3 \right)\\
&&+q^{-5/26}\left(2 +2q+2q^2 -q^3 \right)\\
&&+q^{-15/26}\left( -1 -2q+ 2q^2 -2q^3 \right) + O\left(q^{85/26}\right).
\end{eqnarray*}
These series satisfy the expected identity
\[
\widehat{Z}_{\mathrm{tot}}(q) = \frac{1}{4}\left( \mathsf{Y}_{(1,1), \mathrm{tot}}(q) + \mathsf{Y}_{(1,-1), \mathrm{tot}}(q) + \mathsf{Y}_{(-1,1), \mathrm{tot}}(q) + \mathsf{Y}_{(-1,-1), \mathrm{tot}}(q) \right).
\]


\section*{Acknowledgments} 
The authors' interest on invariant series for $3$-manifolds was sparked by \cite{gukov2021two, akhmechet2023lattice, park2020higher}.  
They thank Sergei Gukov for an email correspondence on this topic and Louisa Liles for related discussions. Additional thanks to Pavel Putrov, Slava Krushkal, and Ross Akhmechet for conversations and correspondences on related topics.
The idea to use reduced plumbing trees is borrowed from \cite{ri2023refined}, and the Weyl assignments from \S\ref{sec:chamberass} were inspired by a similar concept in the $A_1$ case that originated from there.
During the preparation of this manuscript,
AHM was partially supported by the NSF award DMS-2204148, and
NT was partially support by NSF award DMS-2404896 and a Simons Foundation's Travel Support for Mathematicians gift.

\bibliographystyle{alphanumN}
\bibliography{Biblio}

\newcommand{\etalchar}[1]{$^{#1}$}
\begin{thebibliography}{GPPV}

\bibitem[AJK]{akhmechet2023lattice}
R. Akhmechet, P.~K. Johnson, and V. Krushkal.
\newblock Lattice cohomology and {$q$}-series invariants of 3-manifolds.
\newblock {\em Journal f{\"u}r die reine und angewandte Mathematik (Crelles
  Journal)}, 2023(796):269--299, 2023.

\bibitem[AJP]{akhmechet2024knot}
R. Akhmechet, P.~K. Johnson, and S. Park.
\newblock Knot lattice homology and {$q$}-series invariants for plumbed knot
  complements.
\newblock {\em Preprint, arXiv:2403.14461}, 2024.

\bibitem[Bou]{MR1890629}
N. Bourbaki.
\newblock {\em Lie groups and {L}ie algebras. {C}hapters 4--6}.
\newblock Elements of Mathematics (Berlin). Springer-Verlag, Berlin, 2002.
\newblock Translated from the 1968 French original by Andrew Pressley.

\bibitem[GM]{gukov2021two}
S. Gukov and C. Manolescu.
\newblock A two-variable series for knot complements.
\newblock {\em Quantum Topology}, 12(1), 2021.

\bibitem[GPPV]{gukov2020bps}
S. Gukov, D. Pei, P. Putrov, and C. Vafa.
\newblock {BPS} spectra and 3-manifold invariants.
\newblock {\em Journal of Knot Theory and Its Ramifications}, 29(02):2040003,
  2020.

\bibitem[GPV]{gukov2017fivebranes}
S. Gukov, P. Putrov, and C. Vafa.
\newblock Fivebranes and 3-manifold homology.
\newblock {\em Journal of High Energy Physics}, 2017(7):1--82, 2017.

\bibitem[HM]{hirasawa2006genera}
M. Hirasawa and K. Murasugi.
\newblock Genera and fibredness of {M}ontesinos knots.
\newblock {\em Pacific journal of mathematics}, 225(1):53--83, 2006.

\bibitem[HNS]{harichurn2025delta}
S. Harichurn, A. {N\'emethi}, and J. Svoboda.
\newblock Delta invariants of plumbed manifolds.
\newblock {\em SIGMA. Symmetry, Integrability and Geometry: Methods and
  Applications}, 21:091, 2025.

\bibitem[Hum]{MR0323842}
J.~E. Humphreys.
\newblock {\em Introduction to {L}ie algebras and representation theory},
  volume Vol. 9 of {\em Graduate Texts in Mathematics}.
\newblock Springer-Verlag, New York-Berlin, 1972.

\bibitem[LM]{liles2023infinite}
L. Liles and E. McSpirit.
\newblock Infinite families of quantum modular 3-manifold invariants.
\newblock {\em Communications in Number Theory and Physics, to appear,
  arXiv:2306.14765}, 2023.

\bibitem[MT]{MT2}
A.~H. Moore and N. Tarasca.
\newblock A gluing formula for root lattices and plumbed knot complements.
\newblock {\em Preprint}, 2025.

\bibitem[Mur]{murakami2023proof}
Y. Murakami.
\newblock A proof of a conjecture of {Gukov-Pei-Putrov-Vafa}.
\newblock {\em Preprint, arXiv:2302.13526}, 2023.

\bibitem[N{\'e}m]{MR4510934}
A. N{\'e}methi.
\newblock {\em Normal surface singularities}, volume~74 of {\em Ergebnisse der
  Mathematik und ihrer Grenzgebiete. 3. Folge. A Series of Modern Surveys in
  Mathematics}.
\newblock Springer, 2022.

\bibitem[Neu1]{neumann1981calculus}
W.~D. Neumann.
\newblock A calculus for plumbing applied to the topology of complex surface
  singularities and degenerating complex curves.
\newblock {\em Transactions of the American Mathematical Society},
  268(2):299--344, 1981.

\bibitem[Neu2]{neumann2006invariant}
W.~D. Neumann.
\newblock An invariant of plumbed homology spheres.
\newblock In {\em Topology Symposium Siegen 1979: Proceedings of a Symposium
  Held at the University of Siegen, June 14--19, 1979}, pages 125--144.
  Springer, 2006.

\bibitem[Par]{park2020higher}
S. Park.
\newblock Higher rank {$\widehat{Z}$ and $F_K$}.
\newblock {\em SIGMA. Symmetry, Integrability and Geometry: Methods and
  Applications}, 16:044, 2020.

\bibitem[PBI{\etalchar{+}}]{przytycki2023lectures}
J.~H. Przytycki, R.~P. Bakshi, D. Ibarra, G. Montoya-Vega, and D. Weeks.
\newblock {\em Lectures in Knot Theory}.
\newblock Springer, 2023.

\bibitem[Rei]{reidemeister1935homotopieringe}
K. Reidemeister.
\newblock Homotopieringe und {L}insenr{{\"a}}ume.
\newblock In {\em Abhandlungen aus dem Mathematischen Seminar der
  Universit{\"a}t Hamburg}, volume~11, pages 102--109. Springer, 1935.

\bibitem[Ri]{ri2023refined}
S.~J. Ri.
\newblock Refined and generalized {$\hat{Z}$} invariants for plumbed
  3-manifolds.
\newblock {\em SIGMA. Symmetry, Integrability and Geometry: Methods and
  Applications}, 19:011, 2023.

\end{thebibliography}

\end{document}